\documentclass{article}
\usepackage{tikz}
\usepackage{pgfplots}

\usepackage[english]{babel}
\usepackage{ifsym}
\usepackage{cite}
\usepackage{verbatim}
\usepackage[ansinew]{inputenc} %[utf8x]{inputenc}
\usepackage{amsmath, amsthm, amssymb}
\usepackage{nicefrac}
\usepackage{hyperref, graphicx}
\usepackage{geometry}
\usepackage{mathrsfs, stmaryrd, dsfont, nicefrac}
\usepackage{graphicx}
\usepackage{paralist}
\usepackage{subcaption}
\usepackage{enumitem}
\usepackage[linesnumbered,ruled]{algorithm2e}

\newtheorem{definition}{Definition}
\newtheorem{remark}[definition]{Remark}
\newtheorem{theorem}[definition]{Theorem}
\newtheorem{proposition}[definition]{Proposition}
\newtheorem{corollary}[definition]{Corollary}
\newtheorem{lemma}[definition]{Lemma}
\newtheorem{example}[definition]{Example}

\DeclareMathOperator*{\argmin}{arg\, min}

\DeclareMathOperator{\Diag}{Diag}

\newcommand{\R}{\mathbb{R}}

\newcommand{\ignore}[1]{}

\providecommand{\keywords}[1]
{
  \small	
  \textbf{\textit{Keywords: }} #1
}

\providecommand{\ams}[1]
{
  \small	
  \textbf{\textit{AMS subject classifications: }} #1
}

\title{Solving two-parameter eigenvalue problems using an alternating method}
\author{Henrik Eisenmann\thanks{
Max Planck Institute for Mathematics in the Sciences, 04103 Leipzig, Germany; henrik.eisenmann@mis.mpg.de
} \and Yuji Nakatsukasa\thanks{
Mathematical Institute, University of Oxford, Oxford OX2 6GG, UK; nakatsukasa@maths.ox.ac.uk}}
\begin{document}

\maketitle
\begin{abstract}
We present a new approach to compute selected eigenvalues and eigenvectors of the two-parameter eigenvalue problem. Our method requires computing generalized eigenvalue problems of the same size as the matrices of the initial two-parameter eigenvalue problem. The method is applicable for right definite problems, possibly after performing an affine transformation. This includes a class of Helmholtz equations when separation of variables is applied.  We provide a convergence proof for extremal eigenvalues and empirical evidence along with a local convergence proof for other eigenvalues.
\end{abstract}


\ams{65F15, 15A18, 15A69}
\normalsize
\section{Introduction}\label{setting}

In this work we consider the \emph{two-parameter eigenvalue problem}
\begin{equation}\label{2paraev}
\begin{aligned}
(A_1+\lambda  B_1+\mu C_1)u&=0,\\
(A_2+\lambda  B_2+\mu C_2)v&=0
\end{aligned}
\end{equation}
with matrices $A_1,B_1,C_1\in\R^{n\times n}$ and $A_2,B_2,C_2\in\R^{m\times m}$. A solution to this problem is given by possibly complex $(\lambda  ,\mu,u,v)$ if they fulfill \eqref{2paraev} and $u,v\neq 0$. We call the pair $(\lambda  ,\mu)$ an \emph{eigenvalue of the two-parameter eigenvalue problem} if $u,v\neq 0$ exist such that $(\lambda  ,\mu,u,v)$ satisfies \eqref{2paraev} and we call the tensor product $u\otimes v$ an \emph{eigenvector of the two parameter eigenvalue problem} (in the literature the pair $(u,v)$ is often referred to as the eigenvector; the terminology here is simply for convenience).

Two-parameter eigenvalue problems  have been extensively studied~\cite{Atkinson72,MR973644} and naturally arise in mathematical physics when separation of variables is applied. Consider for example the Helmholtz equation 
\begin{align*}
    \Delta u+\lambda  u&=0 \quad\text{in }  \Omega,\\
    u&=0 \quad\text{on } \partial\Omega
\end{align*}
where $\Omega=\{(x,y)\in \R^2:\left(\frac{x}{a}\right)^2+\left(\frac{y}{b}\right)^2<1, y>0\}$ is half of an open ellipse. Using elliptical coordinates, the problem can be reformulated into 
\begin{align}
\begin{aligned}\label{ex: elliptic}
     v''(r)+(\lambda  c^2\sinh^2(r)+\mu) v(r)&=0, \quad v(0)=0=v(R)\\
    w''(\varphi)+ (\lambda  c^2\sin^2(\varphi)-\mu )w(\varphi)&=0,  \quad w(0)=0=w(\pi)
\end{aligned}
\end{align}
with $u(c\cosh(r)\cos(\varphi),c\sinh(r)\sin(\varphi))=v(r)w(\varphi)$ and $\mu \in \R$ (this will be explained in Section \ref{examples}). This is of the form~\eqref{2paraev} after discretization. There are also many other applications that lead to~\eqref{2paraev}, including delay differential equations \cite{jarlebring2009polynomial} and optimization \cite{Shinsaku16}.

There are several numerical methods for solving this problem. 
The traditional approach~\cite{Atkinson72} solves an $n^2\times n^2$ generalized eigenvalue problem \eqref{matrixev} to find all solutions.
Another possibility is Jacobi-Davidson type methods discussed in~\cite{JacobiDavidson02,JacobiDavidson04}. These methods work well in finding eigenvalues close to a given target value. In~\cite{SilArnTwopara} a Sylvester-Arnoldi type method is described that can be used to find a small subset of eigenvalues and uses that Sylvester equations can be solved efficiently. Another possibility is  based on homotopy continutation, for example discussed in~\cite{Continuation2001,Homotopy16}. These aim to find all eigenvalues.

We present a new algorithm that can be seen as an \emph{alternating method} as they are for example considered in~\cite{Holtz2012} for tensors. A similar approach is used in \cite{Bailey_alternating} and \cite{Xingzhi_alternating}. In these methods they alternatingly solve the first eigenvalue problems for $\lambda $ while fixing $\mu$ and the second eigenvalue problem for  $\mu$ while fixing $\lambda $, and repeating the process. Another similar approach was taken in~\cite{Nonlineartwopara}. In there it is described how to transform a two-parameter eigenvalue problem with $n\gg m$ into a nonlinear eigenvalue problem and using techniques for nonlinear eigenvalue problems. By contrast, our method will choose either $\lambda $ or $\mu$ and solve two coupled linear eigenvalue problems alternatingly. This aims to reduce the complexity of finding one solution to~\eqref{2paraev}. The complexity is the one of solving a generalized eigenvalue problem with matrices of size $n\times n$. Hence, the number of operations is  $o(n^3)$ for one solution. 
While we only establish global convergence for extremal eigenpairs, empirically our algorithm also finds all solutions, with $o(n^5)$ cost.  It will turn out that this method can find eigenvalues based on their \emph{index}. The index of an eigenvalue of multiparameter eigenvalue problems is a generalisation of ordering real eigenvalues of a standard symmetric eigenvalue problems~\cite[Ch.~1]{MR973644}.  We will need the following assumptions to hold:
\begin{enumerate}[label=\textbf{A.\arabic*}]
    \item \label{Ass1} All matrices are symmetric; 
    \item The matrices  $C_2$ and $C_1\otimes B_2-B_1\otimes C_2$ are positve definite and $C_1$ is negative definite. \label{Ass2}
\end{enumerate}
Notice that in example \eqref{ex: elliptic} the second derivative is indeed a symmetric operator, and $\sinh^2(r)>0$ and $\sin^2(\varphi)>0$ almost everywhere. A discrete version of these equations hence leads to a problem of the type \eqref{2paraev} that satisfies the assumptions \ref{Ass1} and \ref{Ass2} up to signs.
\begin{remark}
We only require the matrices $C_1$, $C_2$, and $C_1\otimes B_2-B_1\otimes C_2$ to be definite, however we chose these specific definiteness assumptions for convenience as to not have a list of all cases in the following theorems. 
\end{remark}

\begin{remark}
The assumptions are not restrictive for right definite two parameter eigenvalue problems, that is for two parameter problems with positive or negative definite $C_1\otimes B_2-B_1\otimes C_2$. Indeed we can always find an affine transformation of \eqref{2paraev} satisfying the definiteness assumptions.

First we can always find an affine transformation to make one matrix definite. Assume $C_2$ is indefinite. We can then perform the affine transformation
\begin{align*}
B_1&\leftarrow B_1-\frac{v^\top B_2 v}{v^\top C_2 v} C_1\\
B_2&\leftarrow B_2-\frac{v^\top B_2 v}{v^\top C_2 v} C_2
\end{align*}
with $v$ such that $v^\top C_2 v >0$. The new $B_1$ is negative definite, as the positive definiteness of $C_1\otimes B_2-B_1\otimes C_2$ implies that 
\[
(u^\top C_1u )(v^\top B_2v)-(u^\top B_1u)( v^\top C_2v)>0
\]
 for any $u$ and $v$. If $C_2$ is definite we can do the following in a similar fashion, and if $C_2$ is semidefinite but not definite, then $C_1$ is definite. We may therefore assume $B_1$ to be negative definite.  Now let $u=\argmin \frac{u^\top C_1 u}{u^\top B_1 u}$ be a minimizer. We can perform the affine transformation
\begin{align*}
C_1&\leftarrow C_1-(\frac{u^\top C_1 u}{u^\top B_1 u}-\epsilon) B_1\\
C_2&\leftarrow C_2-(\frac{u^\top C_1 u}{u^\top B_1 u} -\epsilon)B_2
\end{align*}
for a sufficiently small $\epsilon>0$. Now the positive definiteness of $C_1\otimes B_2-B_1\otimes C_2$ and negative definiteness of $B_1$ implies that the new $C_2$ is positive definite and negative definiteness of $B_1$ and minimality of $u$ implies negative definiteness of the new $C_1$.

Note that when replacing $B_1,B_2$ and $C_1,C_2$ by the linear combination of matrices 
\[\tilde B_1=b_1 B_1+b_2 C_1,\quad\tilde B_2=b_1 B_2+b_2 C_2
\]
 and 
\[ 
 \tilde C_1=c_1 B_1+c_2 C_1,\quad\tilde C_2=c_1 B_2+c_2 C_2
\]
and find the eigenvalue $(\tilde\lambda,\tilde\mu)$, we can recover the original eigenvalue $(\lambda,\mu)$ via 
\[
\lambda=b_1\tilde\lambda+c_1\tilde\mu ,\quad\mu=b_2\tilde\lambda+c_2\tilde\mu.
\]
\end{remark}
This article is organized as follows: In Section~\ref{Sec:Main}, we first motivate our method and state useful results from multiparameter eigenvalue theory. Afterwards we discuss fixed point properties of the algorithm and prove convergence for extremal eigenvalues. Finally, we examine the time complexity of our method briefly. In Section~\ref{examples}, we show that a class of boundary value problems satisfies our assumptions if properly discretized. Finally, we exhibit results of numerical experiments in Section~\ref{Sec: Numerics}.

\section{An alternating algorithm for the two-parameter problem}\label{Sec:Main}

In this section, we derive our algorithm and prove local convergence for all eigenvalues and global convergence for extremal eigenvalues. First, let us fix our notation. By $\otimes$ we denote the \emph{tensor product}. For vectors this can be seen as the outer product, i.e., $u\otimes v$ corresponds to the rank-one matrix $uv^\top$. The tensor product of matrices then corresponds to a linear operator acting on matrices, i.e., the product $A\otimes B$ acts on the matrix $X$ by $AXB^\top$. This means that $(A\otimes B) (u\otimes v)=(Au\otimes Bv)$. It can also be seen as the Kronecker product for matrices~\cite[Ch.~12]{golubbook4th}. Then the product of two vectors $u\otimes v$ can be reshaped into a rank-one matrix.

\subsection{Derivation of the algorithm}\label{derivation}

The two-parameter eigenvalue problem \eqref{2paraev} can be reduced to two generalized eigenvalue problems~\cite{Atkinson72}. For this define the operators acting on $n\times m$ matrices
\begin{align}
\begin{aligned}
    M_0=B_1\otimes C_2-C_1 \otimes B_2,\\
    M_1=A_1\otimes C_2-C_1 \otimes A_2,\\
    M_2=B_1\otimes A_2-A_1 \otimes B_2.    
    \end{aligned}
\end{align}
A straightforward computation shows that a solution of \eqref{2paraev} satisfies
\begin{align*}
\begin{aligned}
    M_1 (u\otimes v)+\lambda  M_0(u\otimes v)=0,\\
    M_2 (u\otimes v)+\mu M_0(u\otimes v)=0.    
\end{aligned}
\end{align*}
Hence, solutions of \eqref{2paraev} are also rank-one solutions $X$ of the system of generalized matrix eigenvalue problems 
\begin{align}
\begin{aligned}\label{matrixev}
    M_1(X)+\lambda M_0(X)=0,\\
    M_2(X)+\mu M_0(X)=0,
    \end{aligned}
\end{align}
which are of size $nm\times nm$. Note that the eigenvectors are shared between the two generalized eigenvalue problems.
Also under our assumptions solutions of~\eqref{matrixev} lead to solutions of~\eqref{2paraev}. This is a well known consequence of classical theory~\cite{Atkinson72,MR973644}. We provide a self contained proof for completeness.
\begin{lemma}
Under the assumptions~\ref{Ass1} and~\ref{Ass2}, the eigenvectors of the generalized eigenvalue problem $M_1(X)+\lambda M_0(X)=0$ are either rank one or are a linear combination of rank-one eigenvectors. A rank-one eigenvector is then also an eigenvector of the corresponding two-parameter eigenvalue problem.
\end{lemma}

\begin{proof}
The assumption \ref{Ass2} ensures that the eigenspaces of each subproblem in \eqref{matrixev} is spanned by rank-one matrices. 
To see this, first note that the operator $M_0$ is negative definite, and $M_1$ and $M_2$ are symmetric. Therefore all eigenvalues of $\eqref{matrixev}$ are real. Second, we can without loss of generality assume $-C_1$ and $C_2$ to be identity matrices, else we just transform the matrices to $\Tilde{A_1}=(-C_1)^{-\frac{1}{2}}A_1(-C_1)^{-\frac{1}{2}}$, $\Tilde{B_1}=(-C_1)^{-\frac{1}{2}}B_1(-C_1)^{-\frac{1}{2}}$, etc., and $\Tilde{u}=(-C_1)^\frac{1}{2}u$, $\Tilde{v}=C_2^\frac{1}{2}v$. Then $M_0$ and $M_1$ are operators in the form of a \emph{Sylvester equation}, i.e., they have the form
\begin{align*}
 M_0=B_1\otimes I_m+I_n \otimes B_2,\\
    M_1=A_1\otimes I_m+I_n \otimes A_2.  
\end{align*}
A solution to \eqref{matrixev} then satisfies
\begin{align*}
   \left( (A_1+\lambda B_1)\otimes I_m+I_n\otimes (A_2+\lambda B_2)\right)(X)=0,
\end{align*}
i.e., $X$ is a zero solution of a Sylvester equation. Such operators possess an orthonormal basis of rank-one eigenvectors, namely $u\otimes v$, where $u$ is an eigenvector of $(A_1+\lambda B_1)$ and $v$ of  $(A_2+\lambda  B_2)$ \cite[Theorem 4.4.5]{Horn_Topics}, both of which are real symmetric and therefore possess an orthonormal basis of eigenvectors. It follows that $X$ is a sum of rank-one eigenvectors corresponding to the eigenvalue $0$. The dimension equals that of the null space of $M_1(X)+\lambda M_0(X)$, i.e., the geometric multiplicity of $\lambda $.

We next show that a rank-one solution to just one of the eigenvalue problems in \eqref{matrixev} suffices to get a solution of the initial problem~\eqref{2paraev}. Indeed, assume that $X=u\otimes v$ solves the second eigenvalue problem in \eqref{matrixev}. We then get
\begin{align*}
    0=&M_1 (u\otimes v)+\lambda  M_0(u\otimes v)\\
    = &A_1u\otimes C_2v-C_1u \otimes A_2v+\lambda B_1u\otimes C_2v-\mu C_1u \otimes B_2v\\
    =&(A_1+\lambda B_1)u\otimes C_2v- C_1u \otimes(A_2+\lambda B_2)v,
\end{align*}
which holds true if and only if there is a $\mu$ such that 
\begin{align*}
   - \mu C_1u&=(A_1+\lambda  B_1)u,\\
  -  \mu C_2v&=(A_2+\lambda B_2)v.
\end{align*}
This implies \eqref{2paraev}. 
\end{proof}

We can conclude that \eqref{2paraev} has essentially $n m$ solutions which can be obtained by computing the rank-one solutions of one of the eigenvalue problems in \eqref{matrixev}. In the following, we select the first of these eigenvalue problems.

The assumptions imply that the operators $M_0$ and $M_1$ are symmetric and $M_0$ is negative definite. Thus, the solution of 
\begin{align}
    M_1(X)+\lambda M_0(X)=0 \label{2meigenvalue}
\end{align}
with \emph{maximal} eigenvalue $\lambda $ can be obtained by maximizing the \emph{Rayleigh quotient}
\begin{align}\label{rquotient}
    \mathfrak{R}(X)=\frac{\langle X,M_1(X) \rangle}{\langle X,-M_0(X) \rangle},
\end{align}
and since the solution is a rank-one matrix, we can just maximize $ \mathfrak{R}(u\otimes v)$ over $u$ and $v$. For convenience, define the functions
\begin{align*}
    a_1(u)&=u^\top A_1 u,&\quad b_1(u)&=u^\top B_1 u, &\quad c_1(u)&=u^\top C_1 u,\\
    a_2(v)&=v^\top A_2 v,&\quad b_2(v)&=v^\top B_2 v, &\quad c_2(v)&=v^\top C_2 v.
\end{align*}
The assumption~\ref{Ass2} assures that  $b_2(v)C_1-c_2(v) B_1$ and $ c_1(u) B_2-b_1 (u)C_2$ are positive definite.
We can then write
\begin{align*}
\mathfrak{R}(u\otimes v) =-\frac{u^\top (c_2(v) A_1-a_2(v) C_1)u}{u^\top (c_2(v) B_1-b_2(v) C_1)u}.
\end{align*}
For fixed $v$ the matrix in the denominator is negative definite. Hence, the maximal value is given by the maximal eigenvalue of the generalized eigenvalue problem
\begin{align}
    \left(c_2(v) A_1-a_2(v) C_1\right)u=\lambda  (b_2(v) C_1-c_2(v) B_1)u,\label{ev1}
\end{align}
and respectively fixing $u$, the maximal value of
\begin{align*}
\mathfrak{R}(u\otimes v) =-\frac{v^\top (a_1(u)C_2-c_1(u)A_2 )v}{v^\top (b_1(u)C_2 -c_1(u)B_2 )v}
\end{align*}
 is given by the maximal eigenvalue of 
\begin{align}
    \left(a_1(u)C_2 -c_1(u)A_2\right)v=\lambda  (c_1(u)B_2-b_1(u)C_2 )v.\label{ev2}
\end{align}
These are generalized eigenvalue problems with matrices of size $n\times n$ and $m\times m$, while~\eqref{2meigenvalue} is a generalized eigenvalue problem of size $nm\times nm$.

A similar alternating procedure is of course obtained when minimizing the Rayleigh quotient in \eqref{rquotient}, i.e., when aiming at the smallest eigenvalue of \eqref{2meigenvalue}. More generally, nothing even prevents us from updating $u$ and $v$ with non-extremal eigenpairs of the subproblems \eqref{ev1} and \eqref{ev2} in the hope of finding non-extremal eigenpairs of \eqref{2meigenvalue}. For instance, we can solve \eqref{ev1} for 
the $i$-th eigenvalue and \eqref{ev2} for the $j$-th eigenvalue, for some fixed $i$ and $j$.
This idea motivates Algorithm \ref{algo}.  

We will call a pair of eigenvectors $u$ and $v$ a \emph{fixed point of Algorithm~\ref{algo}} if it simultaneously solves the eigenvalue problems~\eqref{ev1} and \eqref{ev2}. Then choosing the corresponding index $(i,j)$ the algorithm will not change $u$ and $v$ anymore, provided eigenvalues are simple.
\begin{algorithm}[!tb]
\caption{Alternating Algorithm for solving two-parameter eigenvalue problems.}\label{algo}
    \SetKwInOut{Input}{Input}
    \SetKwInOut{Output}{Output}

    \Input{Matrices $A_i,B_i,C_i$ for $i=1,2$ satisfying \ref{Ass1} and \ref{Ass2} and index $(i,j)$.}
    \Output{Eigenvalue $(\lambda ,\mu )$ of index $(i,j)$ with corresponding eigenvector $u\otimes v$.}
    select random nonzero $u_0\in\R^{n}$\;
    \For{k=1,2,3,\dots}{
        $a_1:=a_1(u_{k-1}),\quad b_1:=b_1(u_{k-1}),\quad c_1:=c_1(u_{k-1})$\;
        compute the eigenvector $v_k$ corresponding to the $j$-th smallest eigenvalue of the symmetric right definite generalized eigenvalue problem\begin{align*}
           \left(a_1 C_2 -c_1 A_2\right)v=\lambda  ( c_1 B_2-b_1 C_2)v;
        \end{align*}\\
        $a_2:=a_2(v_k);\quad b_2:=b_2(v_k);\quad c_2:=c_2(v_k)$\;
        compute the eigenpair $(u_k,\lambda_k )$ corresponding to the $i$-th smallest eigenvalue of the symmetric right definite generalized eigenvalue problem\begin{align*}
            \left(c_2 A_1-a_2  C_1\right)u=\lambda  (b_2 C_1-c_2 B_1)u;
        \end{align*}\\
    }
	$\lambda:=\lambda_k$, \quad $u:=u_k$,\quad $v:=v_k$,\quad
    $\mu :=-\frac{a_2+\lambda b_2}{c_2}$\;
 
    \KwRet{$(\lambda ,\mu) $ and $u\otimes v$.}
\end{algorithm}

\subsection{Fixed point properties}

Our aim in this section is to show that in principle we can find all rank-one eigenpairs of the problem~\eqref{2meigenvalue}, and hence of~\eqref{2paraev}, by finding fixed points of the subproblems~\eqref{ev1} and~\eqref{ev2} solved in Algortihm~\ref{algo} for all possible input indices $(i,j)\in\{1,\dots,n\}\times\{1,\dots,m \}$. The first step in this direction is the following lemma, which shows
 that a fixed point of the algorithm  indeed provides a solution to~\eqref{2meigenvalue} and~\eqref{2paraev}. 
\begin{lemma}\label{lem: fixed point}
Let $u$ and $v$ simultaneously solve \eqref{ev1} and \eqref{ev2}. Under the assumptions {\upshape \ref{Ass1}} and {\upshape \ref{Ass2}} both eigenvalues coincide. Moreover, setting $\mu :=-\frac{a_2(v)+\lambda b_2(v)}{c_2(v)}$,  $(u\otimes v,\lambda )$ is an eigenpair of \eqref{2meigenvalue} and $(u,v,\lambda ,\mu )$ is a solution for the two-parameter problem \eqref{2paraev}.
\end{lemma}
\begin{proof}
Denote the eigenvalue in \eqref{ev1}  and \eqref{ev2} by $\hat \lambda$ and $\tilde\lambda$, respectively. Then multiplying \eqref{ev1} by $u^\top$ and \eqref{ev2} by $v^\top$, we get
\begin{align*}
    c_2(v)a_1(u)-a_2(v)c_1(u)=-\hat\lambda (c_2(v)b_1(u)-b_2(v)c_1(u))
\end{align*}
and
\begin{align*}
    c_2(v)a_1(u)-a_2(v)c_1(u)=-\tilde\lambda (c_2(v)b_1(u)-b_2(v)c_1(u)).
\end{align*}
Since \ref{Ass2} implies $c_2(v)b_1(u)-b_2(v)c_1(u)<0$, we have $\hat \lambda=\tilde\lambda=:\lambda $. Collecting terms in \eqref{ev1} and \eqref{ev2} gives
\begin{align*}
    c_2(v)A_1u +c_2\lambda  B_1u -(a_2+\lambda  b_2(v)) C_1 u&=0,\\
    -c_1(u) A_2 v -c_1\lambda B_2 v+ (a_1 +\lambda b_1(u)) C_2v&=0.
\end{align*}
Dividing the first equation by $c_2(v)>0$ and the second by $-c_1(u)>0$, we get \eqref{2paraev} with $\mu =-\frac{a_1(u)+\lambda  b_1(u)}{c_1(u)}$, if $$\frac{a_2(v)+\lambda  b_2(v) }{c_2(v)}=\frac{a_1(u)+\lambda  b_1(u) }{c_1(u)}.$$
This equation is however just a consequence of
\begin{align*}
    c_2(v)a_1(u)-a_2(v)c_1(u)=-\lambda (c_2(v)b_1(u)-b_2(v)c_1(u)).
\end{align*}
By the considerations in Section \ref{derivation} $(u\otimes v,\lambda )$ is also a solution of \eqref{2meigenvalue}.
\end{proof}

The previous lemma does not yet let us conclude that all solutions of \eqref{2paraev} occur as fixed points of Algorithm \ref{algo} when varying the input index $(i,j)$. To show this, we  need to introduce the notion of the \emph{index of an eigenvalue} $(\lambda ,\mu )$ of the two-parameter eigenvalue problem.
\begin{definition}\label{defindex}
An eigenvalue $(\lambda ,\mu )$ of the two-parameter eigenvalue problem has the index $(i,j)$ if $0$ is the $i$-th smallest eigenvalue of $A_1+\lambda  B_1+ \mu  C_1$  and the $j$-th smallest eigenvalue of $A_2+\lambda  B_2+ \mu  C_2$.  
\end{definition}
If $0$ is a multiple eigenvalue of $A_1+\lambda  B_1+ \mu  C_1$ or $A_2+\lambda  B_2+ \mu  C_2$, then the corresponding eigenvalue of the two-parameter eigenvalue problem has multiple indices as well.
Under the assumption \ref{Ass1} every real valued eigenvalue of the two-parameter eigenvalue problem has an index since the matrices  $A_1+\lambda  B_1+ \mu  C_1$ and $A_2+\lambda  B_2+ \mu  C_2$  are symmetric.
More important for us is the following result, which immediately follows from~\cite[Theorem 1.4.1]{MR973644}.
\begin{theorem}\label{thindex}
Under the assumptions {\upshape \ref{Ass1}} and {\upshape \ref{Ass2}} there is a unique eigenvalue $(\lambda , \mu )$ to every index $(i,j)\in\{1,\dots,n\}\times \{1,\dots,m\}$.
\end{theorem}

The idea is now to show that the index  in the sense of Definition \ref{defindex}  of the eigenvalue solution provided by a fixed point of Algorithm \ref{algo} coincides with the given input index $(i,j)$. Together with Theorem \ref{thindex} this then implies that all solutions can be obtained this way.
For this we need the following version of Sylvester's law of inertia.

\begin{lemma}\label{evlemma}
Let $A,B$ be symmetric matrices and $I+B$ be positive definite. Then $\lambda _i$ is the $i$-th largest eigenvalue of $A$ if and only if it is the $i$-th largest eigenvalue of the generalized eigenvalue problem
\begin{align*}
    (A+\lambda _i B)u=\lambda  (I+B) u.
\end{align*}
\end{lemma}
\begin{proof}
Consider the matrices $M_1=A-\lambda _i I$ and $M_2=(I+B)^{-\frac{1}{2}}(A-\lambda _i I)(I+B)^{-\frac{1}{2}}$. The matrix $M_2$ is well defined since $I+B$ is positive definite and since $(I+B)^{-\frac{1}{2}}=((I+B)^{-\frac{1}{2}})^T)$ is invertible, $M_1$ and $M_2$ are congruent and therefore have the same number of positive and respectively negative eigenvalues by Sylvester's law of inertia \cite[Theorem 4.5.8]{MR2978290}. 

We can rewrite $M_2$ in the following way:
\begin{align*}
M_2&=(I+B)^{-\frac{1}{2}}(A-\lambda _i I)(I+B)^{-\frac{1}{2}}\\
&=(I+B)^{-\frac{1}{2}}(A+\lambda _i B-\lambda _iB-\lambda _i I)(I+B)^{-\frac{1}{2}}\\
&=(I+B)^{-\frac{1}{2}}(A+\lambda _i B)(I+B)^{-\frac{1}{2}}-\lambda _i I.
\end{align*}
This implies that $A$ and $(I+B)^{-\frac{1}{2}}(A+\lambda _i B)(I+B)^{-\frac{1}{2}}$ have the same number of eigenvalues that are smaller or respectively larger than $\lambda _i$. Since the eigenvalues of $(I+B)^{-\frac{1}{2}}(A+\lambda _i B)(I+B)^{-\frac{1}{2}}$ are the same as the eigenvalues of the generalized eigenvalue problem
\begin{align*}
 (A+\lambda _i B)u=\lambda  (I+B) u
\end{align*}
the claim is proven.
\end{proof}
We are now in a position to prove the main result.
\begin{theorem}\label{mainth}
Let $(i,j)$ be the input index of Algorithm \ref{algo} and let $(u,v,\lambda ,\mu)$ be a fixed point of Algorithm \ref{algo}. Then the output $(\lambda ,\mu )$ is the eigenvalue of problem \eqref{2paraev} with index $(i,j)$ in the sense of Defintion \ref{defindex} and  $u\otimes v$ is the corresponding eigenvector. 
\end{theorem}

\begin{proof}
Let $(u,v)$ be a fixed point of Algorithm \ref{algo} and let $(\lambda ,\mu )$ be the corresponding eigenvalue, i.e., $\lambda $ is the $i$-th smallest eigenvalue of the generalized eigenvalue problem
\begin{align*}
     \left(c_2(v) A_1-a_2(v) C_1\right)u=\lambda  (b_2(v) C_1-c_2(v) B_1)u.
\end{align*}
The Assumption~\ref{Ass2} guarantees that $b_2(v) C_1-c_2(v) B_1$ is positive definite, we may therefore apply Lemma~\ref{evlemma}. 
It follows that $\lambda$ is the $i$-th smallest eigenvalue of the matrix 
\[
c_2(v) A_1-a_2(v) C_1 -\lambda  (b_2(v) C_1-c_2(v) B_1)+\lambda I_n=c_2(v)( A_1 +\lambda B_1 + \mu C_1)+\lambda I_n,
\]
where we substituted $a_2(v)=-\lambda  b_2(v)-\mu  c_2(v)$. This implies that $0$ is the $i$-th smallest eigenvalue of $A_1 +\lambda B_1 + \mu C_1$ since $c_2(v)>0$ is positive by Assumpotion~\ref{Ass2}.

Similarly, $\lambda $ is the $j$-th smallest eigenvalue of the generalized eigenvalue problem
\begin{align*}
     \left(a_1(u) C_2-c_1(u) A_2\right)v=\lambda  (c_1(u) B_2-b_1(u) C_2)v.
\end{align*}
We can again use Lemma~\ref{evlemma} to conclude that $\lambda$ is the $j$-th smallest eigenvalue of the matrix
\[
a_1(u) C_2-c_1(u) A_2-\lambda  (c_1(u) B_2-b_1(u) C_2)+\lambda I_m=-c_1(u)( A_2+\lambda B_2 + \mu C_2)+\lambda I_m,
\]
where we substituted  $a_1(u)=-\lambda  b_1(u)-\mu  c_1(u)$. Note that the values for $\mu$ coincide in a fixed point, as was shown in Lemma~\ref{lem: fixed point}. The Assumption~\ref{Ass2} implies that $c_1(u)<0$ and therefore $0$ is the $j$-th smallest eigenvalue of $A_2+\lambda B_2 + \mu C_2$.
\end{proof}
\begin{remark}
This result made use of the definiteness of $C_1$ and $C_2$. If instead $B_1$ and $B_2$ are definite one obtains a similar correspondence of input indices and indices of the eigenvalue when considering an alternating method resulting from the eigenvalue problem $(M_2+\mu M_0)(X)=0$ instead of $(M_1+\lambda M_0)(X)=0$.
\end{remark}

Theorem~\ref{mainth} implies that if Algorithm \ref{algo} converges, then it computes a solution to the two-parameter eigenvalue problem~\eqref{2paraev}.
\begin{corollary}\label{maincorollry}
Let $(u_k,v_k,\lambda_k)$ be a sequence generated by Algorithm \ref{algo}.  If $u_k$ and $v_k$ converge in the projective sense to $u$ and $v$, i.e., convergence is up to sign flip if we choose normalized $u_k$ and $v_k$ in each step,  then $\lambda_k$ converge to $\lambda $ and $u,v,\lambda $ are fixed by Algorithm~\ref{algo}. Therefore the output $(\lambda ,\mu )$ is the eigenvalue of problem \eqref{2paraev} with index $(i,j)$ in the sense of Defintion \ref{defindex} and $u\otimes v$ is the corresponding eigenvector. 
\end{corollary}
\begin{proof}
Let $X_{2k}=u_k\otimes v_k$ and $X_{2k+1}=u_{k+1}\otimes v_k$. Since $u_k$ and $v_k$ converge to $u$ and $v$ respectively $x_k$ converges to $u\otimes v$.
Define $\hat\lambda_k=\mathfrak{R}(X_{2k+1})$. Notice that $\lambda _k=\mathfrak{R}(X_{2k})$ is the eigenvalue corresponding to~\eqref{ev1} and $\hat\lambda_k$ is the eigenvalue corresponding to~\eqref{ev2} . By continuity of the Rayleigh quotient $\mathfrak{R}$, we get $\lim_{k\to\infty}\lambda _k=\lim_{k\to\infty}\hat\lambda _k=\lambda $. By continuity of the functions $a_1,b_1,c_1,a_2,b_2,c_2$ and continuity of eigenvalues of a nonsingular generalized eigenvalue problem, we get that $u$ and $v$ are eigenvectors of the eigenvalue problems \eqref{ev1} and \eqref{ev2} with eigenvalue $\lambda $ and continuity also ensures that the eigenvalue is still the $i$-th or respectively $j$-th largest. Hence, $(u,v,\lambda ,\mu )$ satisfies the conditions of Theorem~\ref{mainth}.
\end{proof}

\subsection{Geometric interpretation and rate of convergence}

Algorithm~\ref{algo} can be interpreted geometrically. For a given input index $(i,j)$ we look for the intersection of the curves
\begin{align*}
\gamma_i=\{(\lambda,\mu): \text{ $0$ is the $i$-th largest eigenvalue of }A_1+\lambda B_1 +\mu C_1 \},\\
\zeta_j=\{(\lambda,\mu): \text{ $0$ is the $j$-th largest eigenvalue of }A_2+\lambda B_2 +\mu C_2 \}.
\end{align*}
As a consequence of~\cite[Theorem~II.6.1]{Kato_Pert} and definiteness of $C_1$ and $C_2$, these curves are continuous and piecewise analytic, and the corresponding eigenvectors are also piecewise analytic. In a fashion similar to~\cite[Lemma~2.2]{Binding_89} we can derive the tangent line of $\gamma_i$ and $\zeta_j$ at an analytic point $(\lambda,\mu)$ with the corresponding eigenvectors. We can choose $\lambda$ as a local variable and obtain
\begin{align*}
(A_1 +\lambda B_1 +\mu(\lambda) C_1 )u(\lambda)=0
\end{align*}
which implies 
\begin{align*}
(A_1 +\lambda B_1 +\mu(\lambda) C_1 )u'(\lambda)=-(B_1 +\mu'(\lambda) C_1 )u(\lambda).
\end{align*}
Multiplying both equations to with $u(\lambda)^\top$ on the left results in
\begin{align}
a_1(u(\lambda)) +\lambda b_1(u(\lambda)) +\mu(\lambda) c_1(u(\lambda)) &=0,\nonumber\\
b_1(u(\lambda)) +\mu'(\lambda) c_1 (u(\lambda)&=0.\label{eq: eigencurve derivative}
\end{align}
The tangent line $T_\gamma$ of $\gamma_i$ at a point $(\lambda_0,\mu(\lambda_0))$ is therefore given by 
\[
T_\gamma=\{(\lambda,\mu):\mu {c_1(u(\lambda_0))}=-\big({a_1(u(\lambda_0))+\lambda b_1(u(\lambda_0))\big) }\}
\]
 and similarly the tangent line $T_\zeta$ of $\zeta_j$ at a point $(\lambda_0,\mu(\lambda_0))$ is given by 
 \[T_\zeta=\{(\lambda,\mu):\mu{c_2(v(\lambda_0)) =-\big({a_2(v(\lambda_0))+\lambda b_2(v(\lambda_0))\big) }}\},
 \]
 where $v(\lambda_0)$ is the corresponding eigenvector of $A_2+\lambda_0 B_2 +\mu(\lambda_0) C_2 $. Examining Algorithm~\ref{algo}, it can therefore be interpreted as follows:
\begin{enumerate}
\item Start at a point $(\lambda,\mu)$ of $\gamma_i$ and compute it's tangent line $T_\gamma$;
\item compute the intersection point $(\lambda,\mu)$ of $T_\gamma$ and $\zeta_j$  and compute the tangent line $T_\zeta$;
\item compute the intersection point $(\lambda,\mu)$ of $T_\zeta$ and $\gamma_i$  and compute the tangent line $T_\gamma$;
\item go to step 2.
\end{enumerate}
 We can use this interpretation to estimate the local rate of convergence. Notably, this interpretation is related to Newton's method for computing an intersection point of two functions $f$ and $g$. While Newton's method would compute a new iterate via
 \[
 f(x_k)-g(x_k)+(f'(x_k)-g'(x_k))(x_{k+1}-x_k)=0,
 \]
our method performs two steps via
\begin{align}
 f(x_k)+f'(x_k)(x_{k+1}-x_k)=g(x_{k+1})\quad \text{and}\quad g(x_{k+1})+g'(x_{k+1})(x_{k+2}-x_{k+1})=f(x_{k+2}). \label{quasinewton}
\end{align}

This allows us to prove local quadratic convergence in a similar fashion to a standard Newton's method.

\begin{proposition}\label{prop: local convergence}
Let $f,g:U\subset \R\to\R$ be differentiable with Lipschitz continuous  derivatives $f'$ and $g'$ and corresponding constants $\alpha$ and $\beta$. Furthermore, let $|f'(x)-g'(x)|\geq\gamma>0$ for all $x$. Then a sequence $x_k$ generated by~\eqref{quasinewton} satisfies
\begin{align*}
|\Delta x_k|(1-\frac{\beta}{2}|\Delta x_k|) &\leq \frac{\alpha}{2\gamma}|\Delta x_{k-1}|^2,\\
|\Delta x_{k+1}|(1-\frac{\alpha}{2}|\Delta x_{k+1}|) &\leq \frac{\beta}{2\gamma}|\Delta x_{k}|^2,
\end{align*}
where $\Delta x_k=x_{k+1}-x_k$.
\end{proposition}
\begin{proof}
We start with~\eqref{quasinewton} and with a Taylor expansion for $g$ we get
\begin{align*}
 f(x_k)+f'(x_k)\Delta x_k=g(x_{k+1})=g(x_k)+g'(x_k)\Delta x_k +\int_0^1 \big(g'(x_k+s\Delta x_k)-g'(x_k)\big) \Delta x_k ds.
\end{align*}
A  Taylor expansion of $f$ and $g(x_k)=f(x_{k-1})+f'(x_{k-1})\Delta x_{k-1}$ leads to
\[
f(x_k)=g(x_k)+\int_0^1 \big(f'(x_{k-1}+s\Delta x_{k-1})-f'(x_{k-1})\big) \Delta x_{k-1} ds.
\]
Combining both equations, we get
\[
\big(f'(x_k)-g'(x_k)\big)\Delta x_k =\int_0^1 \big(g'(x_k+s\Delta x_k)-g'(x_k)\big) \Delta x_k - \big(f'(x_{k-1}+s\Delta x_{k-1})-f'(x_{k-1})\big) \Delta x_{k-1} ds.
\]
Using Lipschitz continuity and $|f'(x)-g'(x)|>\gamma$, we arrive at
\[
\gamma |x_k|\leq \frac{\beta}{2}|x_k|^2+\frac{\alpha}{2}|x_{k-1}|^2.
\]
The first inequality follows directly, and the second one can be derived analogously.
\end{proof}
  We can apply this result, when the intersection point of $\gamma_i$ and $\zeta_j$ is an analytic point of both curves. Then we can describe both locally via smooth functions $\mu=f(\lambda)$ for $\gamma_i$ and  $\mu=g(\lambda)$ for $\zeta_j$. The difference $|f'-g'|$ can be bounded with~\eqref{eq: eigencurve derivative}. We get $|f'-g'|\geq |\frac{b_1(u)}{c_1(u)}-\frac{b_2(v)}{c_2(v)}|>0$ by the definiteness assumption~\ref{Ass2}.
  
\subsection{Global convergence for extremal eigenpairs}

So far, we have shown that if the algorithm converges, it returns the eigenvalue of the two-parameter eigenvalue problem with the desired index.  For extremal indices,  we can actually prove global convergence. 

\begin{theorem}\label{convtheorem}
The sequences $u_k$ and $v_k$ generated by Algorithm \ref{algo} converge (up to normalization and sign flip) to a solution of~\eqref{2paraev} if the input index is either $(1,1)$ or $(n,m)$ and under the assumption that the respective eigenvalue is simple.
\end{theorem}
We use the following lemma for proving convergence when each step is at least as good as a line search.
\begin{lemma}\label{gradlemma}
Let $f:D\to\R$ be continuously differentiable, where $D\subset\R^n$ is open. Let $\{x_k\}_{k\in \mathbb{N}}\subset K\subset D$ and $g_k=\nabla f(x_k)$ satisfy
\begin{align*}
f(x_{k+1})\leq f(x_k-\sigma g_k)
\end{align*} 
for every $\sigma\in\R$ such that $x_k-\sigma g_k\in D$, where $K$ is compact. Then $g_k$ converges to zero as $k\to\infty$.
\end{lemma}

\begin{proof}[Proof of Lemma~\ref{gradlemma}]
First note that $ f(x_{k+1})\leq f(x_k-\sigma g_k)$ for all $\sigma$ implies that $f(x_k)$ is monotonically nonincreasing and $\{x_k\}_{k\in \mathbb{N}}\subset K$ implies that $f(x_k)$ is bounded from below. Hence, $f(x_k)$ converges to some value $\tilde{f}$ and since $x_k$ lies in a compact set, the gradients $g_k$ are bounded as well.

Now towards a contradiction assume that $\|g_k\|\geq 2\epsilon> 0$ for some $\epsilon > 0$ and every $k$. Else either $g_k$ converges to zero or we can choose a subsequence $g_{k'}$ such that $\|g_{k'}\|\geq 2\epsilon> 0$. Since $g_k$ are bounded, there exists $\delta_1>0$ such that $x_k-\sigma g_k\in \tilde{K}\subset D$ for every $\sigma\in [0,\delta_1]$, where $\tilde{K}$ is also compact. Since $f$ is continuously differentiable, $\nabla f$ is uniformly continuous on $\tilde{K}$. Therefore, there is $\delta_2 \in (0, \delta_1]$ such that
\begin{align*}
\|\nabla f(x_k-\xi g_k)-g_k\|<\epsilon.
\end{align*}
for all $\xi \in [0,\delta_2]$. By the mean value theorem, Cauchy-Schwarz inequality, and the assumption $\|g_k\|\geq2\epsilon$, we have
\begin{align*}
f(x_k) - f(x_k-\delta_2 g_k)&=\delta_2\nabla f(x_k-\xi g_k)^\top g_k\\
&\geq\delta_2 \|g_k\|^2-\delta_2\|\nabla f(x_k-\xi g_k) - g_k\|\|g_k\|\\
&\geq\delta_2\|g_k\|^2-\delta_2\|g_k\|\epsilon\\
&\geq \frac{\delta_2}{2} \| g_k \|^2 \geq2\delta_2\epsilon^2.
\end{align*}
This yields the contradiction
\begin{align*}
\infty>f(x_0)-\tilde{f}=\sum_{k=0}^\infty f(x_k)-f(x_{k+1})\geq\sum_{k=0}^\infty f(x_k)-f(x_k-\delta_2 g_k)\geq  \sum_{k=0}^\infty 2\delta_2\epsilon^2=\infty,
\end{align*}
therefore $g_k$ converges to zero.
\end{proof}

\begin{proof}[Proof of Theorem~\ref{convtheorem}] 
Without loss of generality, assume that $C_1=-I_{n}$ and $C_2=I_{m}$. Finding the vectors $u$ and $v$ for the indices  $(1,1)$ and $(n,m)$ corresponds to either minimizing or maximizing the Rayleigh quotient
\begin{align*}
    \mathfrak{R}(X)=-\frac{\langle X,M_1(X) \rangle}{\langle X,M_0(X) \rangle}.
\end{align*}
Its gradient is given by
\begin{align*}
    \nabla \mathfrak{R}(X)=-\frac{2}{\langle X,M_0(X)\rangle}\left( M_1(X)+\mathfrak{R}(X)M_0(X)  \right)
\end{align*}
and if $X=u\otimes v$, then since $C_1=-I_{n}$ and $C_2=I_{m}$, we have
\begin{align*}
    \nabla \mathfrak{R}(u\otimes v)=-\frac{2}{\langle u\otimes v,M_0(
u\otimes v)\rangle}\left( u\otimes (A_2+\mathfrak{R}(u\otimes v)B_2)v+(A_1+\mathfrak{R}(u\otimes v)B_1)u\otimes v  \right).
\end{align*}
This shows that the gradient is an element in the tangent space  of the rank-one matrix manifold at $u\otimes v$, which is given by
\begin{align*}
    T_{u\otimes v}\mathcal{M}_1=\{x\otimes v+u\otimes y:x\in \R^{n},y\in \R^{m}\}.
\end{align*}
Now assume every iterate of $u$ and $v$ are unit vectors. Then $-c_1(u)=1=c_2(v)$ and the Rayleigh quotient reads
\begin{align*}
    \mathfrak{R}(u\otimes v)=\frac{a_1(u)+a_2(v)}{b_1(u)+b_2(v)}.
\end{align*}
Now let $(\lambda ,v)$ be an eigenpair of
\begin{align*}
            \left(a_1(u)I_{m} +A_2\right)v=-\lambda  (b_1(u)I_{m}+ B_2)v,
\end{align*}
which is one step of Algorithm \ref{algo}. It follows that $\lambda =\mathfrak{R}(u\otimes v)$ and
\begin{align*}
   -  \left(A_2 +\lambda  B_2\right)v=(a_1(u)+\lambda  b_1(u))v.
\end{align*}
The gradient is then 
\begin{align*}
    \nabla \mathfrak{R}(u\otimes v)=\frac{2}{\langle u\otimes v,M_0(
u\otimes v)\rangle}\left(-u\otimes (a_1(u)+\lambda  b_1(u))v+(A_1+\lambda  B_1)u\otimes v  \right),
\end{align*}
and is orthogonal to the subspace $\{u\otimes y:y\in\R^{m}\}\subset  T_{u\otimes v}\mathcal{M}_1$ as 
\begin{align*}
   \langle \nabla \mathfrak{R}(u\otimes v),u\otimes y\rangle=\frac{2}{\langle u\otimes v,M_0(
u\otimes v)\rangle}\left(-  (a_1(u)+\lambda  b_1(u))v^\top y+(a_1(u)+\lambda  b_1(u)) v^\top y  \right)=0.
\end{align*}
This was to be expected, as in each step we optimize the Rayleigh quotient with respect to the subspace $\{u\otimes y:y\in\R^{m}\}$ or respectively $\{x\otimes v:x\in\R^{n}\}$, which makes the gradient orthogonal to that respective space. This, together with $\nabla \mathfrak{R}(u\otimes v)\in T_{u\otimes v}\mathcal{M}_1 $, implies that the gradient lies in $\{u\otimes y:y\in\R^{m}\}$ or respectively $\{x\otimes v:x\in\R^{n}\}$ after each half step.

We now only consider the case of minimizing the Rayleigh quotient, as maximizing can be done analogously. As we can choose our iterates $u_k$ and $v_k$ to be normalized, the iterates $X_k=u_k\otimes v_k$ lie in a compact set. As $$v_{k+1}=\argmin_{\|v\|=1} \mathfrak{R}(u_k\otimes v),$$
we have $\mathfrak{R}(X_{k+1})\leq \mathfrak{R}(X_k-\sigma \nabla \mathfrak{R}(X_k))$ since $X_k-\sigma \nabla \mathfrak{R}(X_k)$ is of the form $u_k\otimes v$, and the $\mathfrak{R}$ is scaling invariant. We can thus use Lemma~\ref{gradlemma} and therefore $\nabla \mathfrak{R}(X^k)$ converges to zero, implying that any convergent subsequence of $X_k$ converges to an eigenvector of $M_1(X)+\lambda  M_0(X)=0$. Corollary~\ref{maincorollry} implies that this eigenvector corresponds to the eigenvalue of the correct index. Hence, since this eigenvalue is simple, the sequence $X_k$ only has one accumulation point up to sign flip.
\end{proof}

\subsection{Complexity}
We discuss time complexity of Algorithm~\ref{algo}. We assume $m=n$. Computing the eigendecomposition of a generalized eigenvalue problem needs $O(n^{\omega+\gamma})\subset o(n^3)$ operations, where $\omega$ is the exponent of complexity of matrix multiplication and $\gamma>0$~\cite{Demmel_Fast_2007}. This implies that the complexity for computing one eigenvalue of the two-parameter eigenvalue problem is $o(n^3 k)$, where $k$ is the number of iterations in Algorithm~\ref{algo} until a sufficient accuracy is achieved (empirically we observe $k=O(1)$). When computing all $n^2$ eigenvalues of the two-parameter eigenvalue problem, we therefore get a complexity of $o(n^5 k)$ (and with $k=O(1)$ we get a complexity of $o(n^5)$). However, a full eigenvalue decomposition is not always necessary. 
For extremal eigenvalues, we only need to compute the eigenvector corresponding to the largest or smallest eigenvalue, which is typically possible in $O(n^2)$ operations, for example using Lanczos or LOBPCG~\cite{lobpcg}. Similarly,  the $i$-th largest eigenvalue can often be computed in $O(n^2i)$ operations. In summary, computing an eigenvalue with index $(i,j)$ needs $O(n^2k\max(\min(i,n-i+1),\min(j,n-j+1))$ operations, where $k$ is the number of iterations necessary for the desired accuracy.
If the matrices allow for fast matrix vector multiplication, for example if they are sparse with $O(n)$ nonzero entries, the complexity can be reduced further up to $O(nk\max(\min(i,n-i+1),\min(j,n-j+1))$ operations.

\section{A class of PDE eigenvalue problems}\label{examples}

We now present a class of PDE eigenvalue problems, which can be separated into an appropriate two-parameter eigenvalue problem. Let $\phi:(a,b)\times(c,d) \to U\subset \R^2$ be a diffeomorphism, with
\begin{align*}
g(x,y):=(D\phi(x,y))^\top D\phi(x,y)=\begin{pmatrix}
g_1(x)+g_2(y)&0\\
0&g_1(x)+g_2(y)
\end{pmatrix},
\end{align*}and choose $g_1,g_2 > 0$. This is always possible for such a function $g$ as $g(x,y)$ is positive definite and therefore $g_1(x)+g_2(y)> 0$. Now let $\Omega=\phi((a,b)\times(c,d))$. We consider the Helmholtz equation
\begin{align*}
    \Delta u+\lambda  u=0 \quad&\text{on } \Omega, \\
    u=0 \quad&\text{at } \partial\Omega.
\end{align*}
We can now write $\Delta u$ in coordinates given by $\phi$. Then 
\begin{align*}
    (\Delta u)(\phi(x,y))=&\frac{1}{\sqrt{\det(g(x,y))}}\nabla\cdot\left(\sqrt{\det(g(x,y))}g^{-1}(x,y) \nabla(u(\phi(x,y))\right)\\
    =&\frac{1}{g_1(x)+g_2(y)}\Delta(u(\phi(x,y))).
\end{align*}

Making an ansatz $u(x,y)=v(x)w(y)$, we get 
\begin{align*}
 0&=   v''(x)w(y)+v(x)w''(y)+\lambda  g_1(x)v(x)w(y)+\lambda  g_2(y)v(x)w(y)\\
 &=v(x)\big(w''(y)   +\lambda  g_1(x)w(y)-\mu w(y) \big)+w(y)\big(v''(x)+\lambda  g_2(y)v(x)+\mu v(x)\big).
\end{align*}
Now let $v$ and $w$ satisfy
\begin{align}\label{pde2para}
    v''(x)+\lambda  g_1(x)v(x)+\mu v(x)&=0\\
    w''(y)+\lambda  g_2(y)w(y)-\mu w(y)&=0,
\end{align}
and $0=v(a)=v(b)=w(c)=w(d)$. Then $u(x,y)=v(x)w(y)$ solves the original eigenvalue problem, and on a given discretrization \eqref{pde2para} satisfies \ref{Ass1} and \ref{Ass2}.
\begin{example}
The complex function $\phi:\mathbb{C}\to \mathbb{C},z\mapsto z^2$. Then $D\phi(z)=2z$. Therefore with $z=x+iy$, we have 
\begin{align*}
    g(x,y)=\begin{pmatrix}2x&2y\\-2y&2x\end{pmatrix}\begin{pmatrix}2x&-2y\\2y&2x\end{pmatrix}=\begin{pmatrix}4x^2+4y^2&0\\0&4x^2+4y^2\end{pmatrix}.
\end{align*}
\end{example}
\begin{example}
The complex function $\phi:\mathbb{C}\to \mathbb{C},z\mapsto e^z$. Then $D\phi(z)=e^z$. Therefore with $z=x+iy$, we have 
\begin{align*}
    g(x,y)=\begin{pmatrix}e^x \cos(y)&e^x \sin(y)\\-e^x \sin(y)&e^x \cos(y)\end{pmatrix}\begin{pmatrix}e^x \cos(y)&-e^x \sin(y)\\e^x \sin(y)&e^x \cos(y)\end{pmatrix}=\begin{pmatrix}e^{2x}&0\\0&e^{2x}\end{pmatrix}.
\end{align*}
These coordinates describe an annulus.
\end{example}
\begin{example}
The complex function $\phi:\mathbb{C}\to \mathbb{C},z\mapsto \cosh(z)$. Then $D\phi(z)=\sinh(z)$. Therefore with $z=x+iy$, we have 
\begin{align*}
    g(x,y)=&\begin{pmatrix}\sinh(x) \cos(y)&\cosh(x) \sin(y)\\-\cosh(x) \sin(y)&\sinh(x) \cos(y)\end{pmatrix}\begin{pmatrix}\sinh(x) \cos(y)&-\cosh(x) \sin(y)\\\cosh(x) \sin(y)&\sinh(x) \cos(y)\end{pmatrix}\\
    =&\begin{pmatrix}\sinh^2(x)+\sin^2(y)&0\\0&\sinh^2(x)+\sin^2(y)\end{pmatrix}.
\end{align*}
These are elliptical coordinates.
\end{example}

\section{Numerical Experiments}  \label{Sec: Numerics}

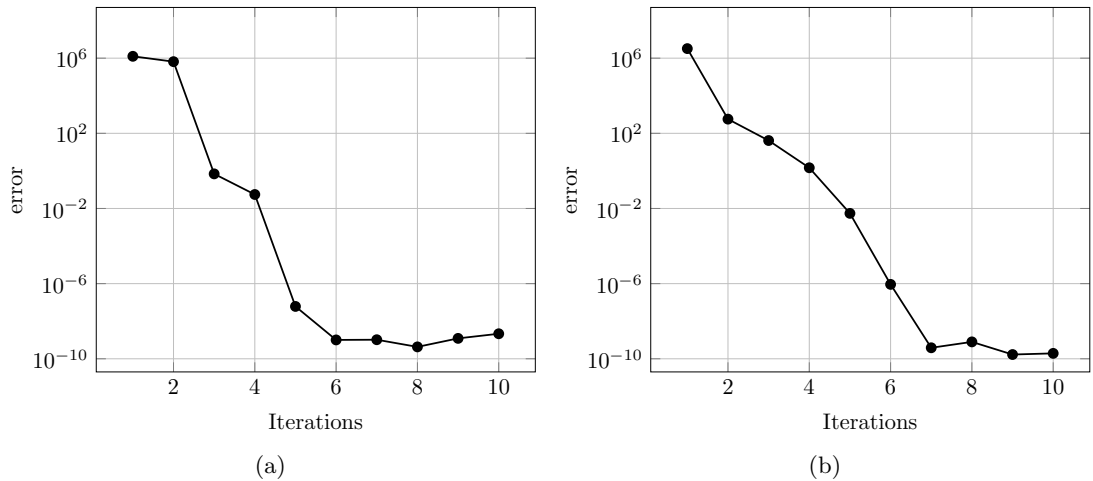
\begin{figure}[t]
\begin{subfigure}{0.48\textwidth}
\resizebox{\textwidth}{!}{
\begin{tikzpicture}
\begin{semilogyaxis}[ymin=2*10^-11,ymax=5*10^8,grid=major,xlabel=Iterations,ylabel=error
%$\|\nabla \mathfrak{R}(u\otimes v)\|$
] 
\addplot[mark=*,,line width=0.8pt
] file[] {data1000gradrand.txt};
\end{semilogyaxis}
\end{tikzpicture}
}
\caption{}\label{grad1000rand}
\end{subfigure}
\begin{subfigure}{0.48\textwidth}
\resizebox{\textwidth}{!}{
\begin{tikzpicture}
\begin{semilogyaxis}[ymin=2*10^-11,ymax=5*10^8,grid=major,xlabel=Iterations,ylabel=error%$\|\nabla \mathfrak{R}(u\otimes v)\|$
]
 
\addplot[mark=*,line width=0.8pt
] file[] {data1000grad.txt};
\end{semilogyaxis}
\end{tikzpicture}
}
\caption{}\label{grad1000}
\end{subfigure}
\caption{Testing Algorithm~\ref{algo} with input index $(1,1)$. Figure~\subref{grad1000rand} depicts results for randomly generated $1000\times 1000$ matrices fulfilling Assumptions~\ref{Ass1} and \ref{Ass2}\ and Figure~\subref{grad1000} depicts results for a discretization of \eqref{ex: elliptic} on {a} $1000\times 1000$ grid. }\label{gradtest}
\end{figure}

%
%\begin{figure}[t]
%\begin{subfigure}{0.48\textwidth}
%\includegraphics[width=\textwidth]{data1.png}
%\caption{}\label{succes1}
%\end{subfigure}
%\begin{subfigure}{0.48\textwidth}
%\includegraphics[width=\textwidth]{data2.png}
%\caption{}\label{succes2}
%\end{subfigure}
%\caption{Success for a randomly generated example with $30\times 30$ matrices satisfying the assumptions by index. White indicates that the alternating algorithm found the eigenvalue for that index, black indicates no success. Figure~\subref{succes1} used $M_1+\lambda M_0$ and \subref{succes2} used $M_2+\mu M_0$.
%}
%\end{figure}

\begin{figure}[t]
         \begin{subfigure}[b]{0.33\textwidth}
        \centering
        \includegraphics[width= \textwidth]{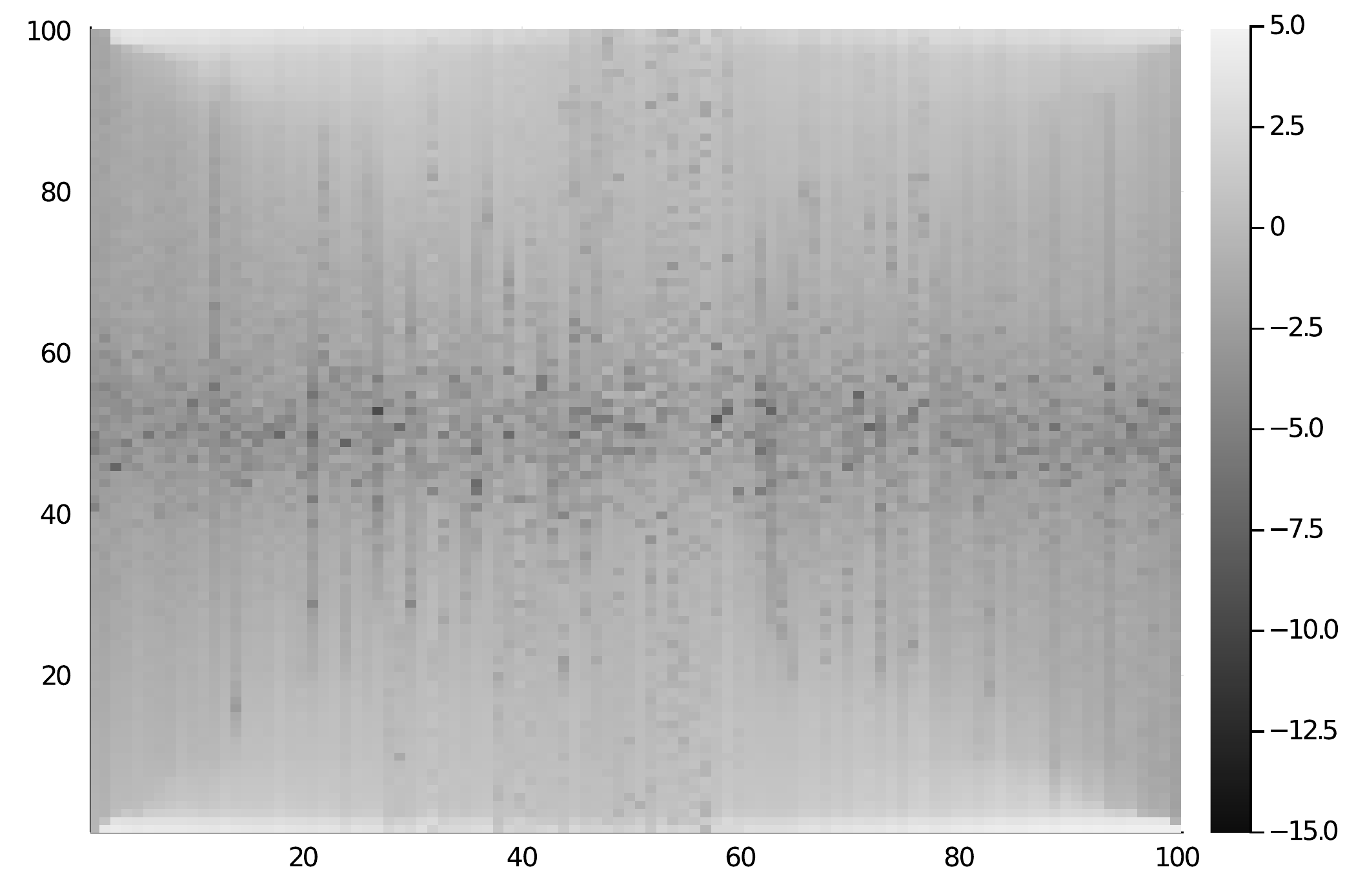}
        \caption{Two iterations}
    \end{subfigure}
    \begin{subfigure}[b]{0.33\textwidth}
        \centering
        \includegraphics[width=\textwidth]{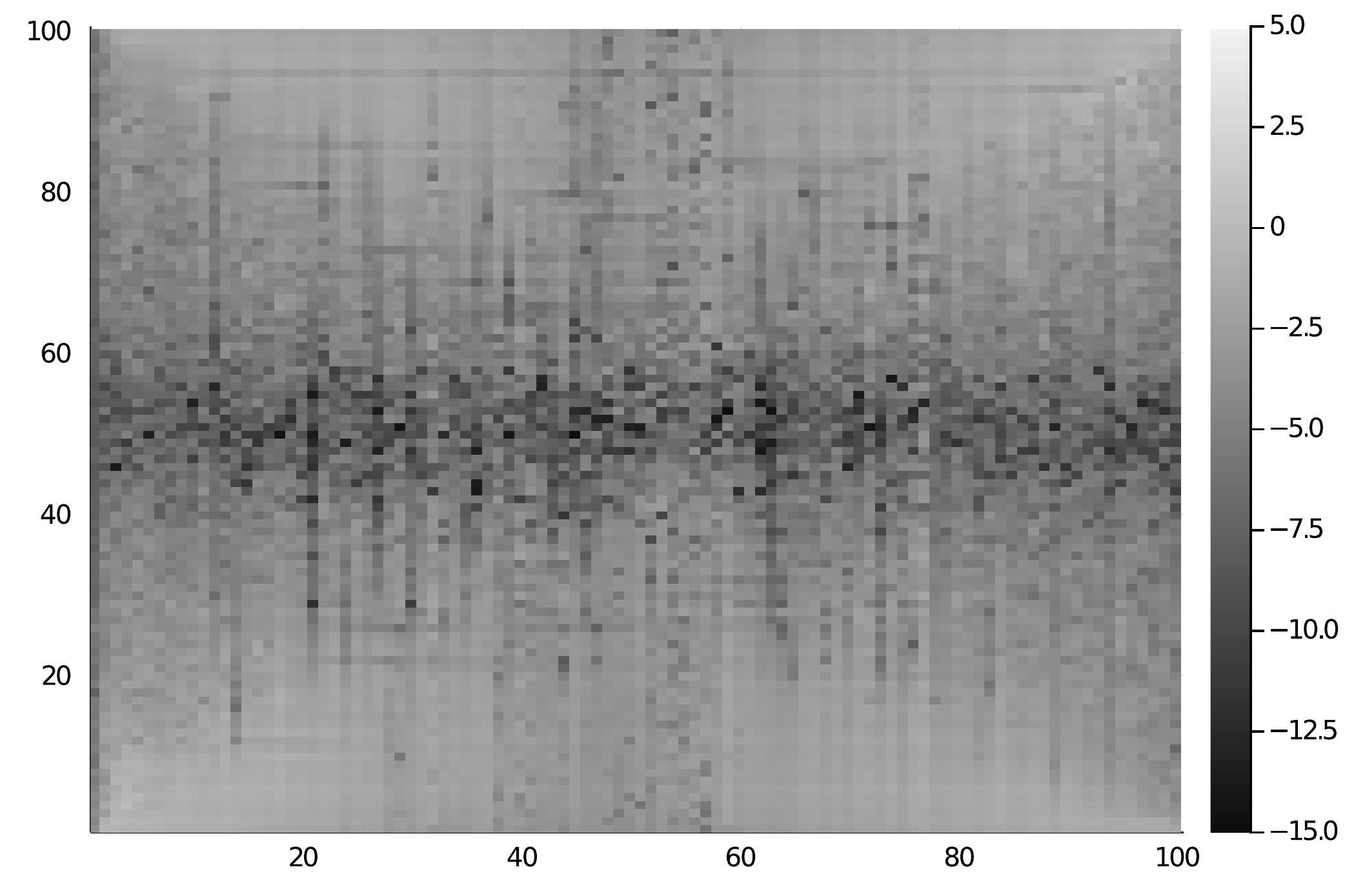}
        \caption{Three iterations}
    \end{subfigure}
     \begin{subfigure}[b]{0.33\textwidth}
        \centering
        \includegraphics[width= \textwidth]{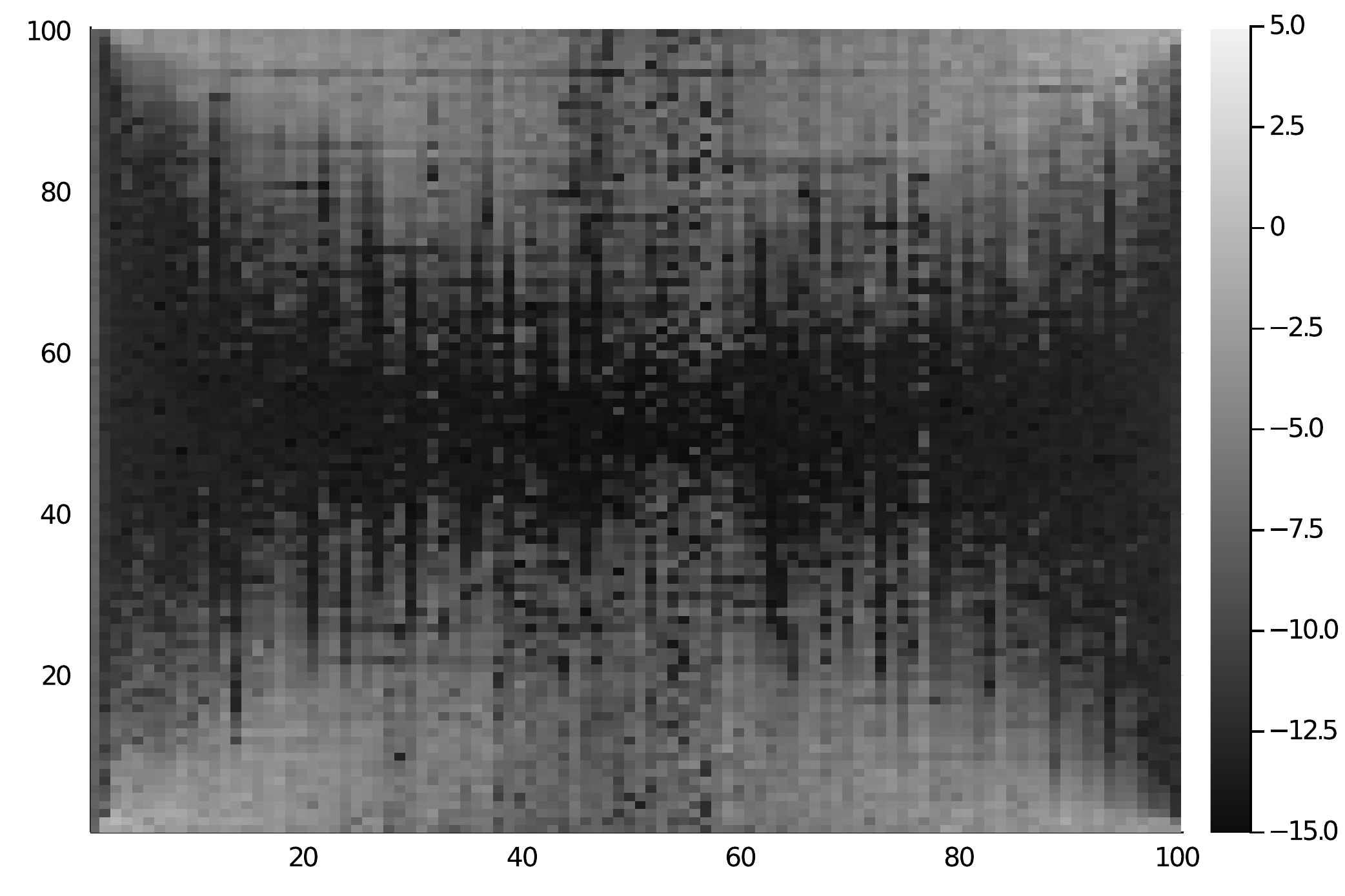}
        \caption{Four iterations}
    \end{subfigure}

     \begin{subfigure}[b]{0.33\textwidth}
        \centering
        \includegraphics[width=\textwidth]{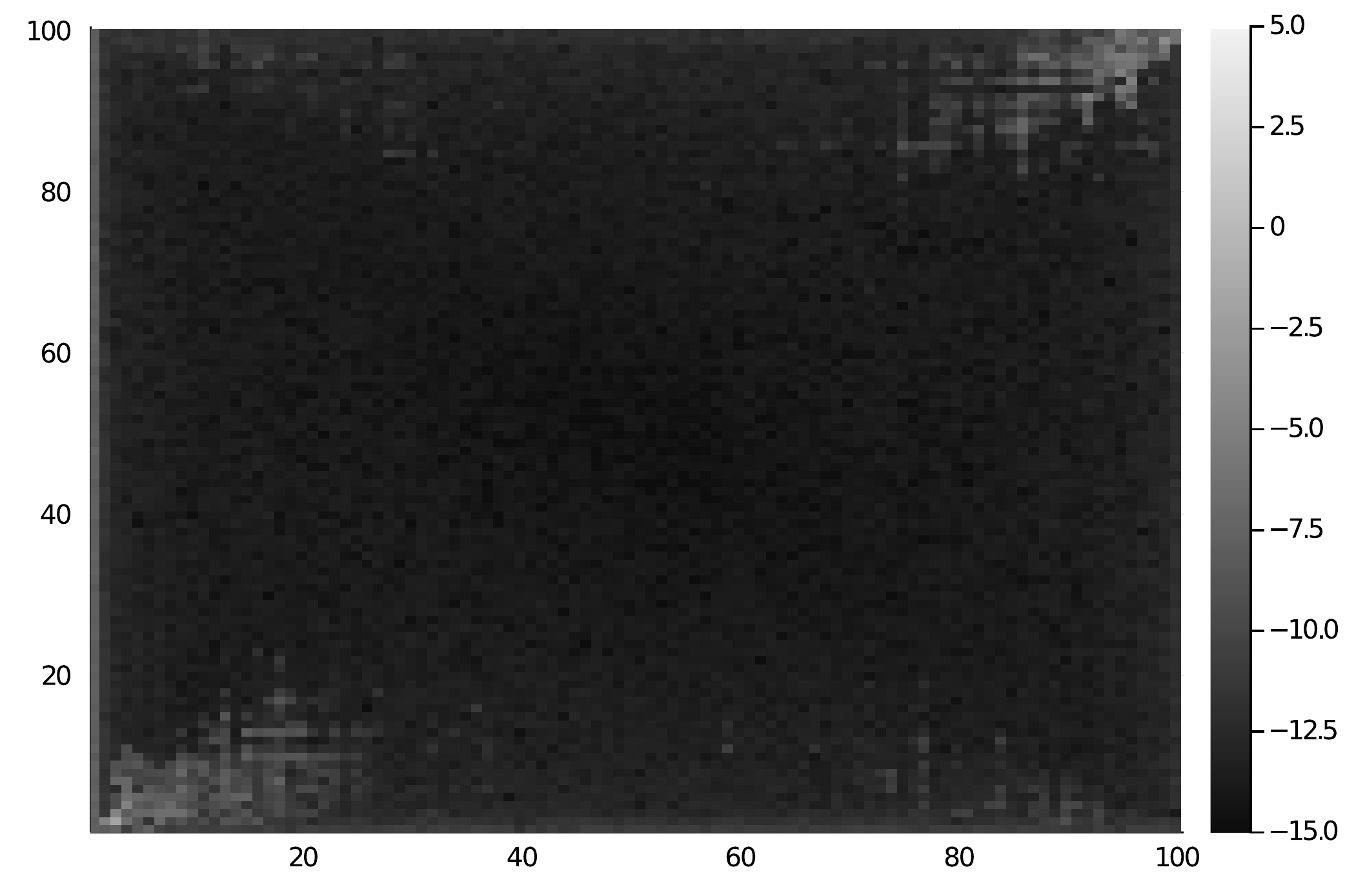}
        \caption{Five iterations}
    \end{subfigure}    
         \begin{subfigure}[b]{0.33\textwidth}
        \centering
        \includegraphics[width= \textwidth]{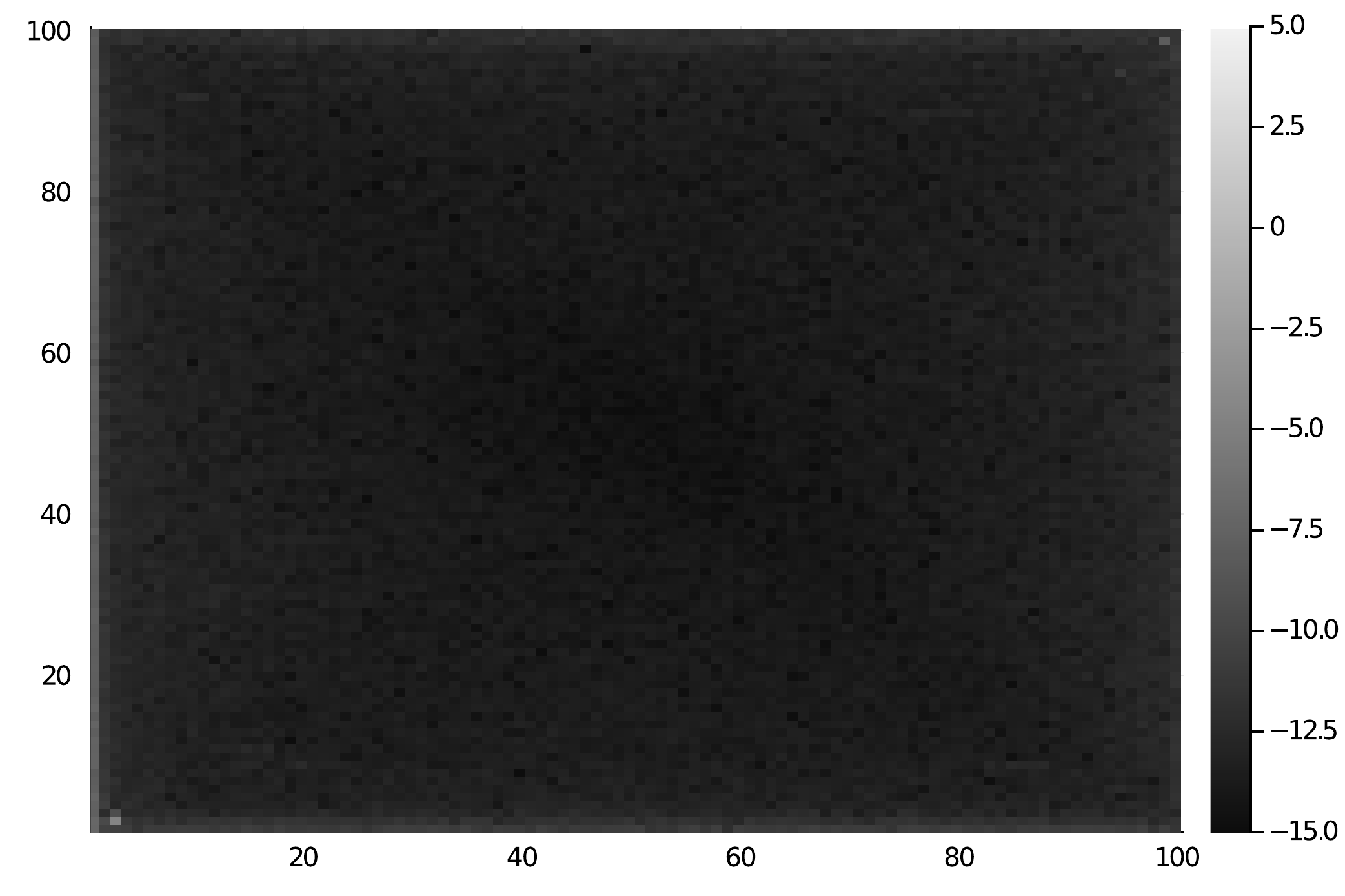}
        \caption{Six iterations}
    \end{subfigure}
     \begin{subfigure}[b]{0.33\textwidth}
        \centering
        \includegraphics[width=\textwidth]{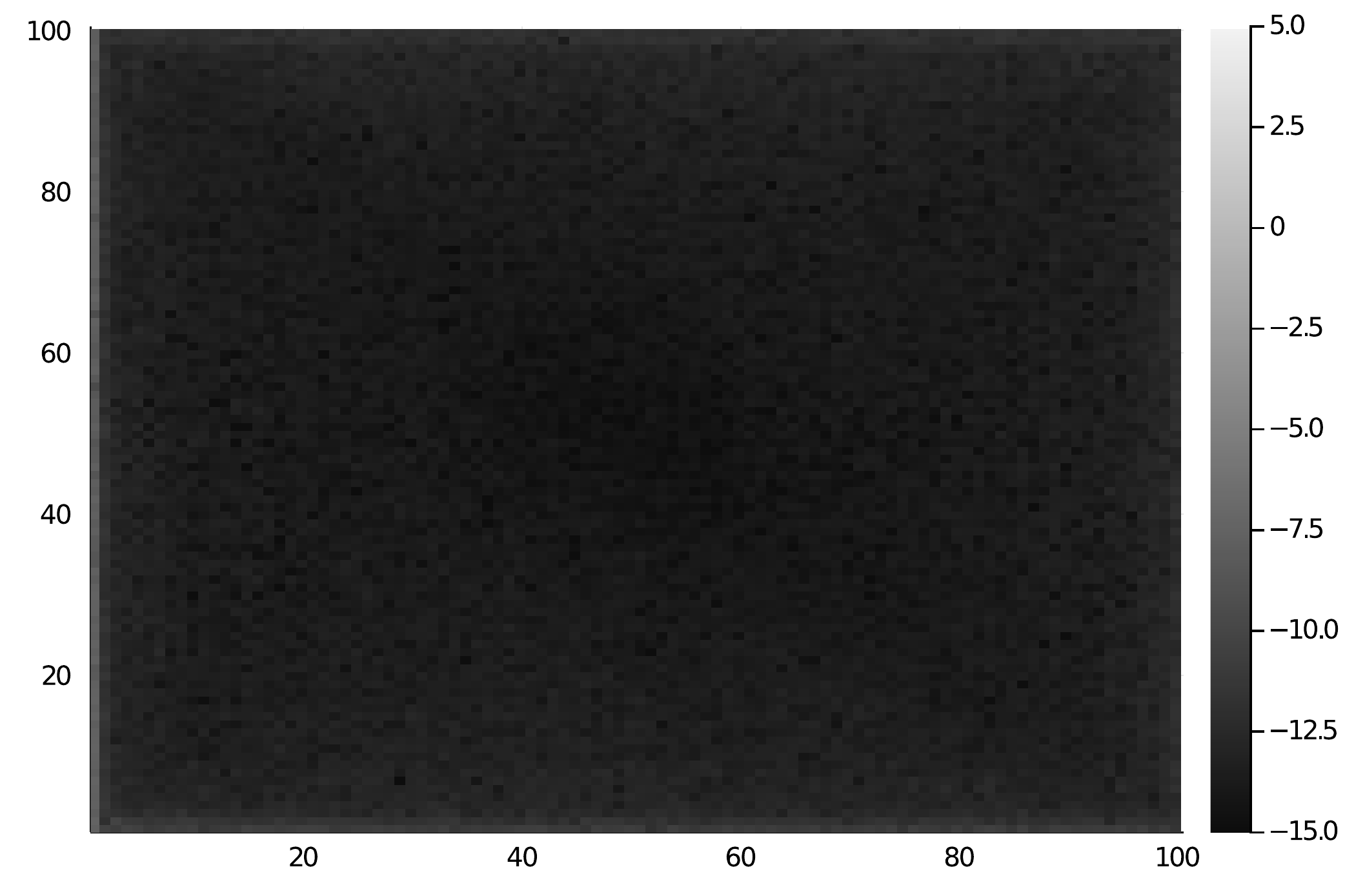}
        \caption{Seven iterations}
    \end{subfigure}    

        \caption{Testing Algorithm \ref{algo} with random $100\times 100$ matrices fulfilling the assumptions. Every picture is taken after solving one eigenvalue problem in line 4 or 6 of the algorithm. The axis label the indices of the computed eigenvalue and the colorscale shows the sum of the absolute values of the $i$-th and $j$-th eigenvalue in~\eqref{2paraev} on a logarithmic scale. }\label{exp: random}
\end{figure}

In this section, we present results from numerical experiments to test the performance of Algorithm~\ref{algo}. For a measure of the error we use Definition~\ref{defindex}. For a given approximate eigenvalue $(\lambda,\mu)$ we compute the $i$-th smallest eigenvalue of $A_1+\lambda B_1+\mu C_1$ and the $j$-th smallest eigenvalue of $A_2+\lambda B_2+\mu C_2$ and take the sum of the respective absolute values. This quantity is zero if and only if $(\lambda,\mu)$ is an eigenvalue with index $(i,j)$. This is also the sum of the corresponding residual norms with the corresponding normalized eigenvectors.

In a first experiment we generated matrices satisfying Assumptions~\ref{Ass1} and~\ref{Ass2} randomly. 
For the matrices $A_1$ and $A_2$ we generated $n\times n$ and $m\times m$ matrices with independent standard Gaussian distributed entries and took the symmetric part.
 For the matrices $B_1,B_2,C_1,C_2$ we generated matrices  $S_1$ and $S_2$ with Gaussian distributed  entries and an $n$ dimensional array $b_1$ with values
  uniformly distributed between $-0.5$ and $0.5$ and an $m$ dimensional array $b_2$ with values uniformly distributed between $-1.5$ and $-0.5$. We then chose the matrices 
\[
B_1=S_1^{}\Diag(b_1)S_1^\top,\quad B_2=S_2^{}\Diag(b_2)S_2^\top,\quad C_1=-S_1^{}S_1^\top,\quad C_2=S_2^{}S_2^\top,
\]
where $\Diag(b_i)$ is the matrix with the entries of $b_i$ on its diagonal.

\begin{figure}[t]
         \begin{subfigure}[b]{0.33\textwidth}
        \centering
        \includegraphics[width= \textwidth]{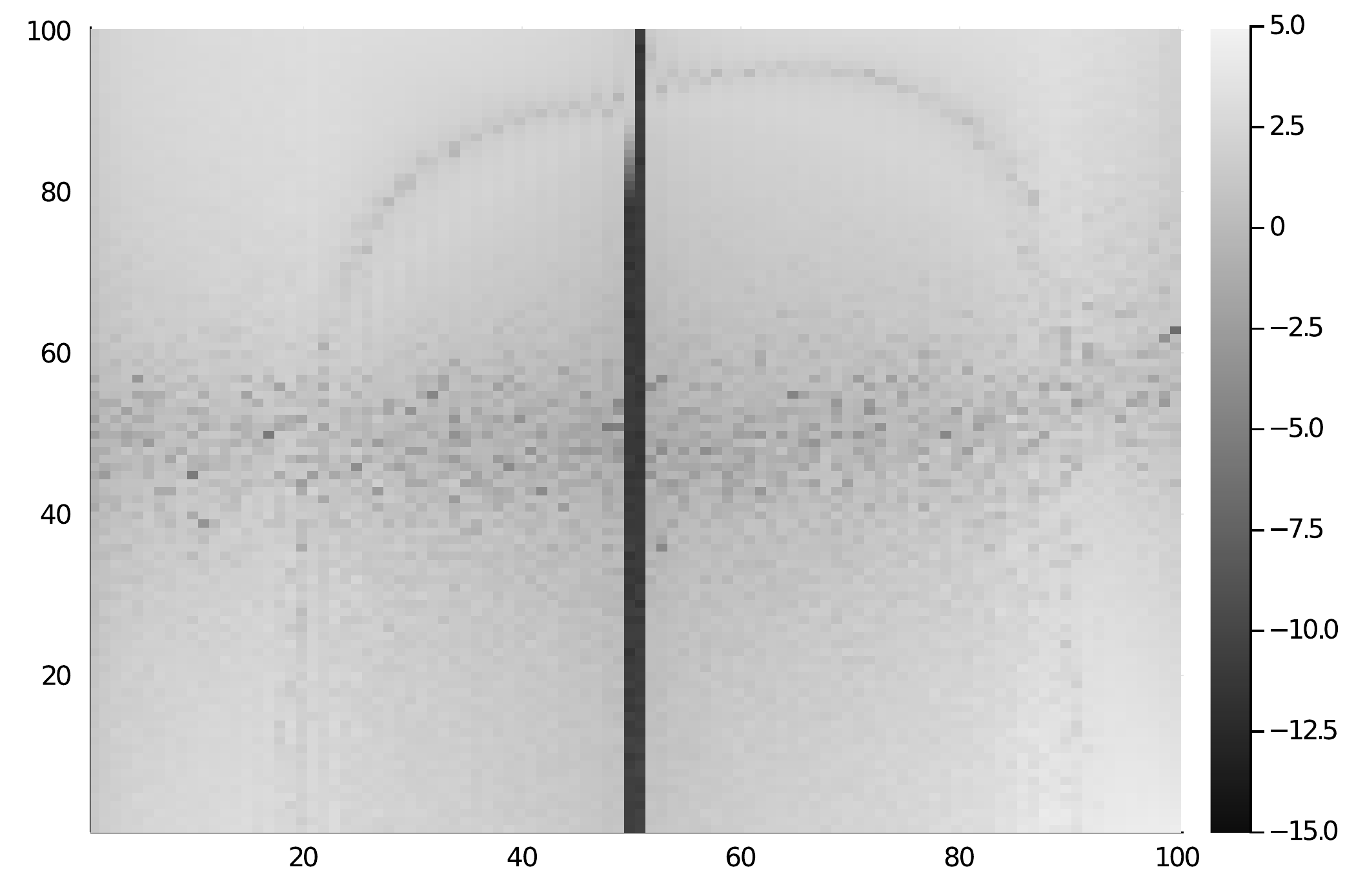}
        \caption{Two iterations}
    \end{subfigure}
    \begin{subfigure}[b]{0.33\textwidth}
        \centering
        \includegraphics[width=\textwidth]{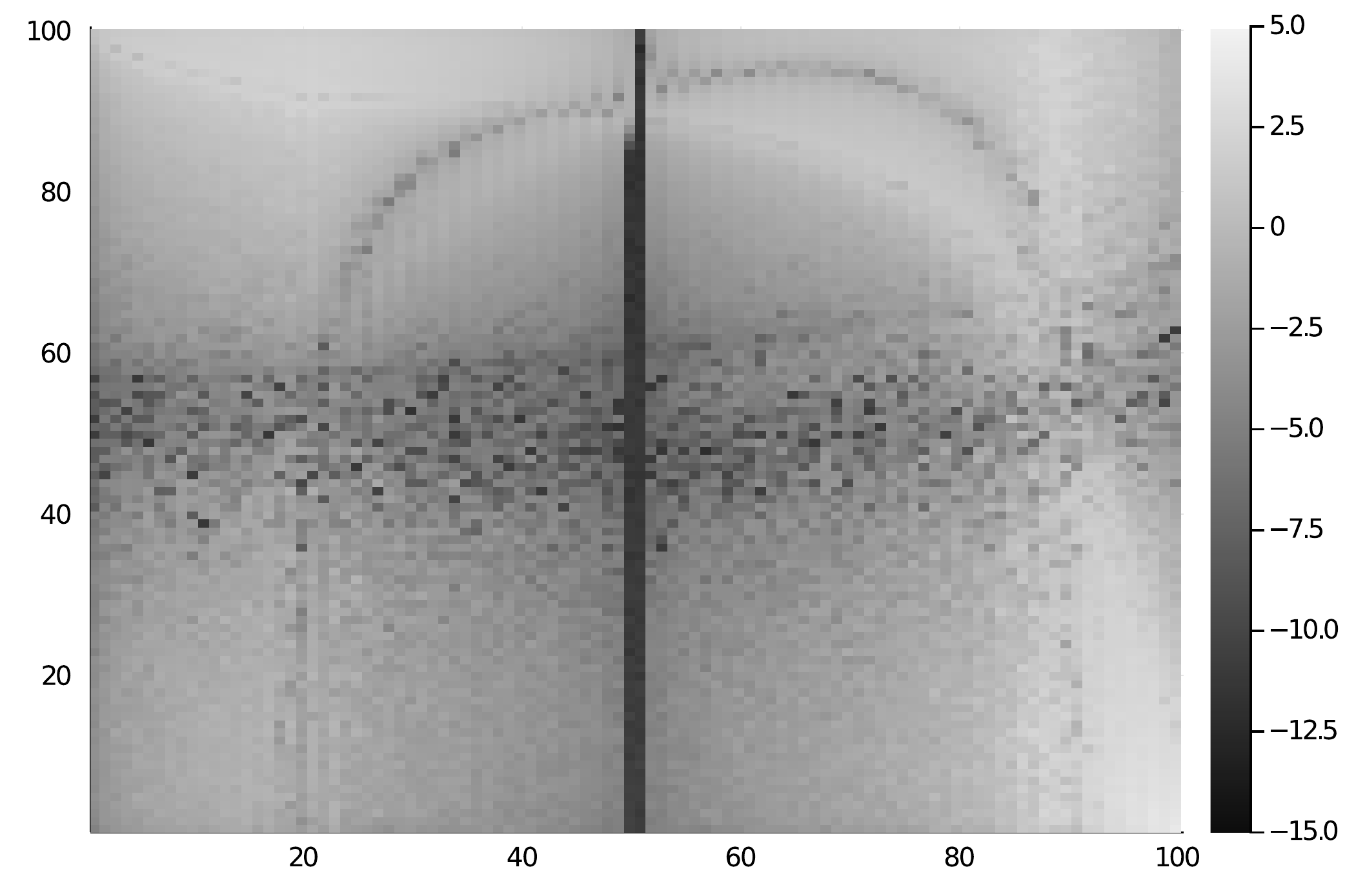}
        \caption{Three iterations}
    \end{subfigure}
     \begin{subfigure}[b]{0.33\textwidth}
        \centering
        \includegraphics[width= \textwidth]{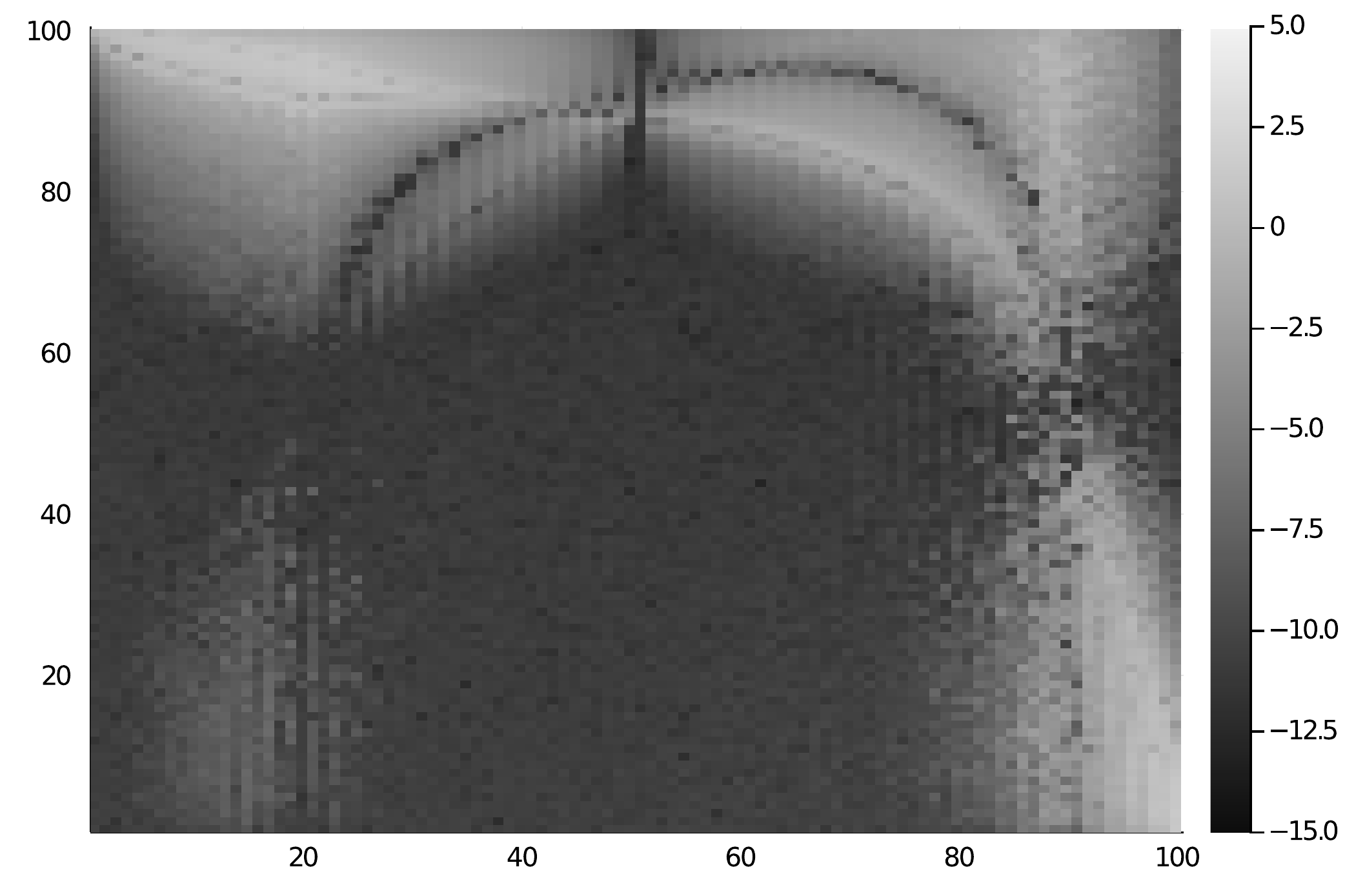}
        \caption{Four iterations}
    \end{subfigure}

     \begin{subfigure}[b]{0.33\textwidth}
        \centering
        \includegraphics[width=\textwidth]{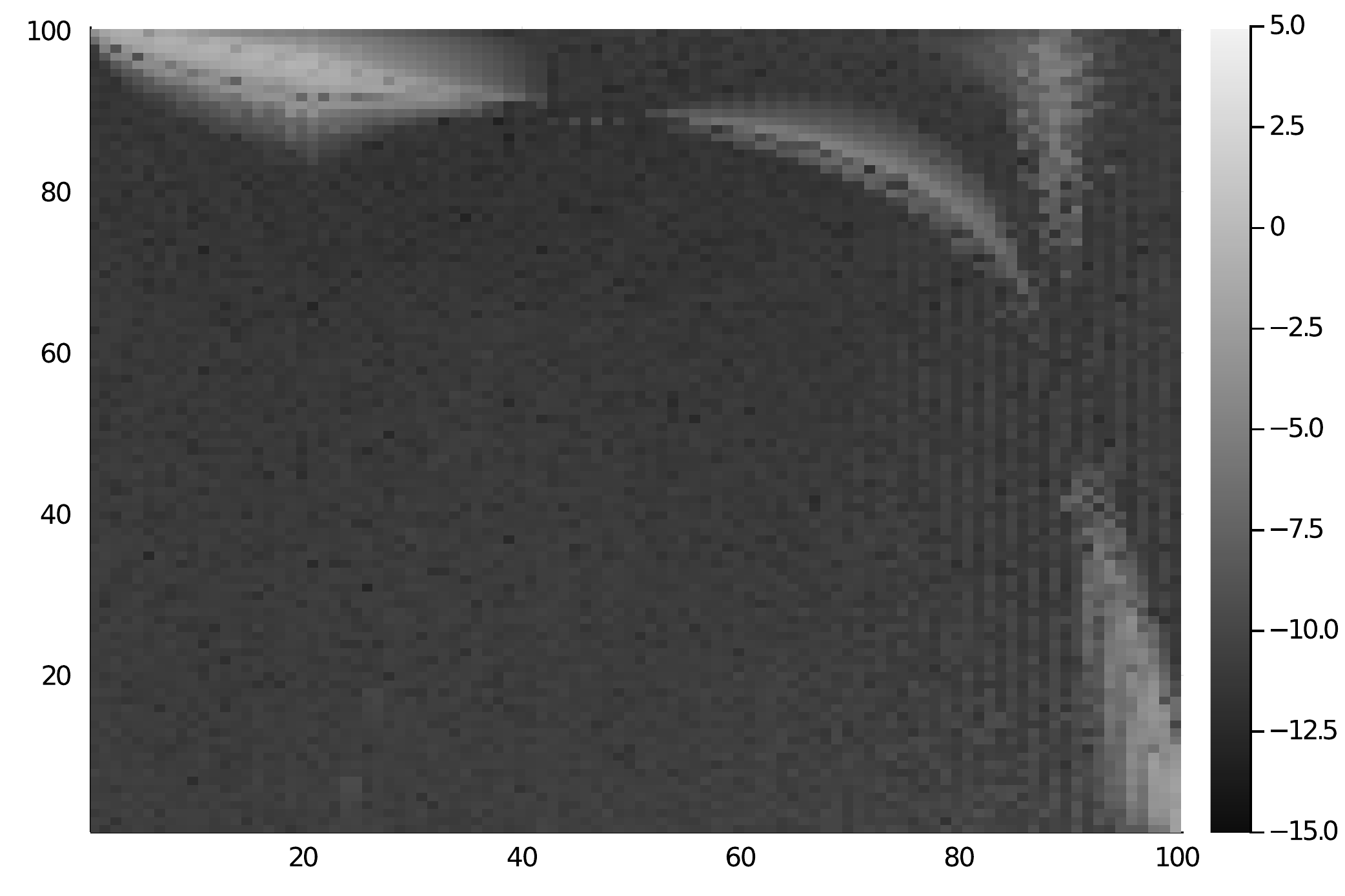}
        \caption{Five iterations}
    \end{subfigure}    
         \begin{subfigure}[b]{0.33\textwidth}
        \centering
        \includegraphics[width= \textwidth]{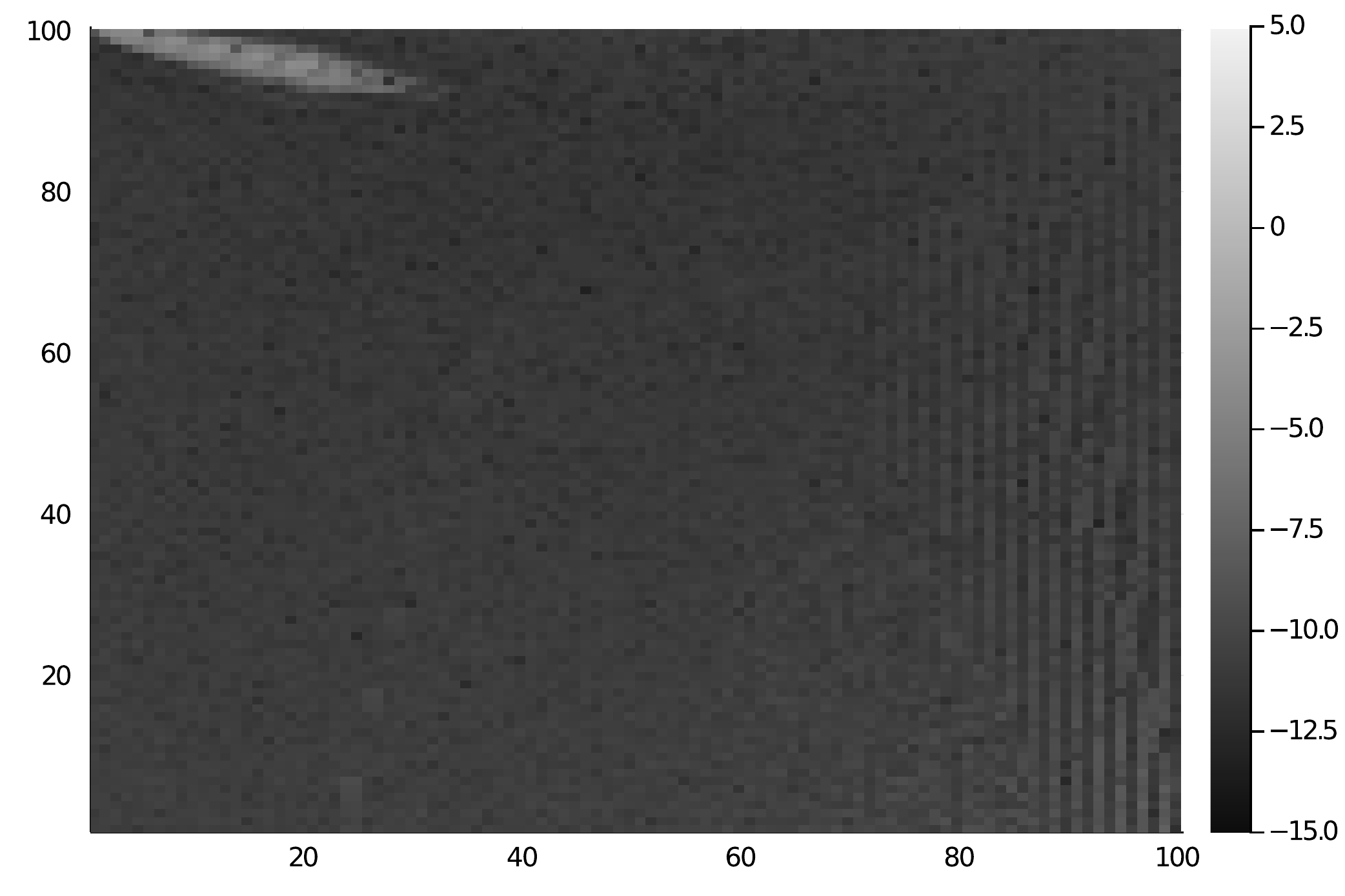}
        \caption{Six iterations}
    \end{subfigure}
     \begin{subfigure}[b]{0.33\textwidth}
        \centering
        \includegraphics[width=\textwidth]{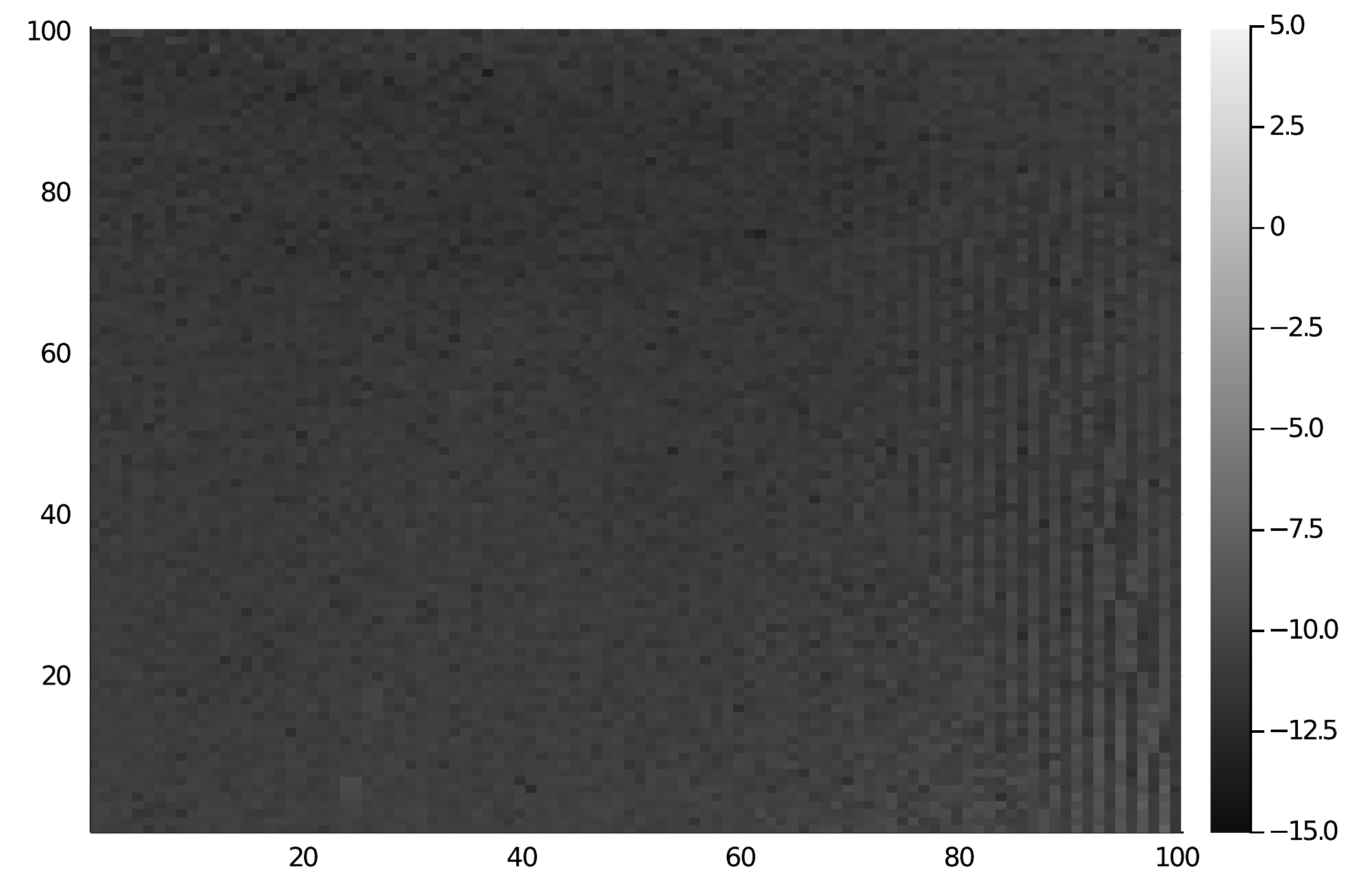}
        \caption{Seven iterations}
    \end{subfigure}    
    \caption{Testing Algorithm \ref{algo} with a discretization of~\eqref{ex: elliptic} on a  $100 \times 100$ grid. Every picture is taken after solving one eigenvalue problem in line 4 or 6 of the algorithm. The axis label the indices of the computed eigenvalue and the colorscale shows the sum of the absolute values of the $i$-th and $j$-th eigenvalue in~\eqref{2paraev} on a logarithmic scale. }\label{exp: ellipse}
\end{figure}

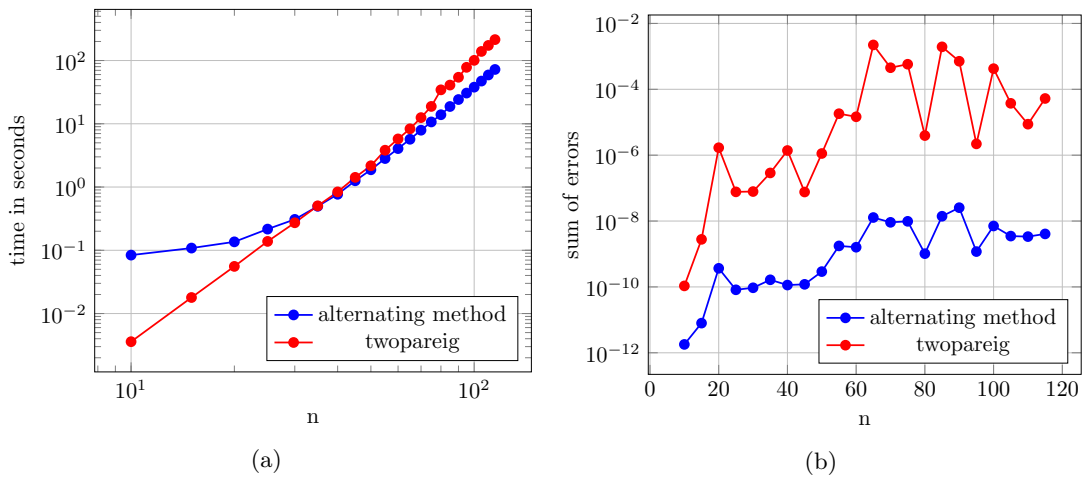
\begin{figure}[t]
\begin{subfigure}{0.48\textwidth}
\resizebox{\textwidth}{!}{
\begin{tikzpicture}
\begin{loglogaxis}[grid=major,xlabel=n,ylabel=time in seconds,
legend pos=south east] 
\addplot[mark=*,color=blue,line width=0.8pt
] file[skip first] {timedata3.txt};
\addlegendentry{alternating method}
\addplot[mark=*,color=red,line width=0.8pt
] file[skip first] {timedata2.txt};
\addlegendentry{twopareig}
\end{loglogaxis}
\end{tikzpicture}
}
\caption{}\label{subfig: times}
\end{subfigure}
\begin{subfigure}{0.48\textwidth}
\resizebox{\textwidth}{!}{
\begin{tikzpicture}
\begin{semilogyaxis}[grid=major,xlabel=n,ylabel=sum of errors,
legend pos=south east] 
\addplot[mark=*,color=blue,line width=0.8pt
] file[skip first] {errsdata.txt};
\addlegendentry{alternating method}
\addplot[mark=*,color=red,line width=0.8pt
] file[skip first] {errsdata2.txt};
\addlegendentry{twopareig}
\end{semilogyaxis}
\end{tikzpicture}
}
\caption{}\label{subfig: errors}
\end{subfigure}
\caption{Comparing the computational time~\subref{subfig: times} and precision~\subref{subfig: errors} for randomly generated examples with matrices of size $n\times n$ satisfying the assumption~\ref{Ass1} and~\ref{Ass2} for our method and \texttt{twopareig} from \cite{twopareig}. }\label{timecomp}
\end{figure}

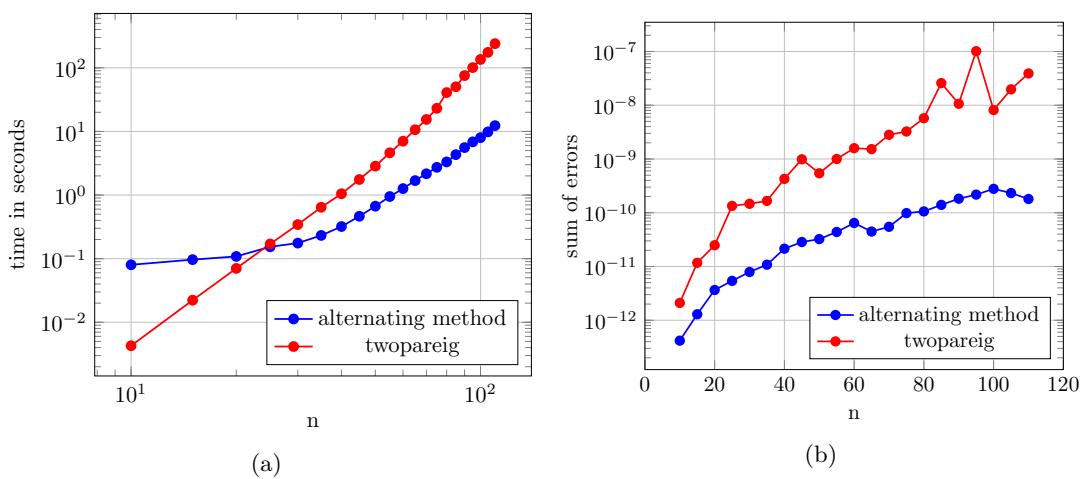
\begin{figure}[t]
\begin{subfigure}{0.48\textwidth}
\resizebox{\textwidth}{!}{
\begin{tikzpicture}
\begin{loglogaxis}[grid=major,xlabel=n,ylabel=time in seconds,
legend pos=south east] 
\addplot[mark=*,color=blue,line width=0.8pt
] file[skip first] {timedatad.txt};
\addlegendentry{alternating method}
\addplot[mark=*,color=red,line width=0.8pt
] file[skip first] {timedatad2.txt};
\addlegendentry{twopareig}
\end{loglogaxis}
\end{tikzpicture}
}
\caption{}\label{subfig: timesdiag}
\end{subfigure}
\begin{subfigure}{0.48\textwidth}
\resizebox{\textwidth}{!}{
\begin{tikzpicture}
\begin{semilogyaxis}[grid=major,xlabel=n,ylabel=sum of errors,
legend pos=south east] 
\addplot[mark=*,color=blue,line width=0.8pt
] file[skip first] {errsdatad.txt};
\addlegendentry{alternating method}
\addplot[mark=*,color=red,line width=0.8pt
] file[skip first] {errsdatad2.txt};
\addlegendentry{twopareig}
\end{semilogyaxis}
\end{tikzpicture}
}
\caption{}\label{subfig: errorsdiag}
\end{subfigure}

\caption{Comparing the computational time~\subref{subfig: timesdiag} and precision~\subref{subfig: errorsdiag} for randomly generated examples with matrices of size $n\times n$ satisfying the assumption~\ref{Ass1} and~\ref{Ass2} for our method and \texttt{twopareig} from \cite{twopareig}. For this example the matrices $B_1,B_2,C_1,C_2$ are diagonal matrices.}\label{timecompdiag}
\end{figure}

As a second example, we discretized the Helmholtz equation on half of an ellipse as in the example in Section \ref{setting}, i.e.,
\begin{align}\label{ellipseforexperiment}
\begin{aligned}
     v''(r)+(\lambda  c^2\sinh^2(r)+\mu) v(r)&=0, \quad v(0)=0=v(1)\\
    w''(\varphi)+ (\lambda  c^2\sin^2(\varphi)-\mu )w(\varphi)&=0,  \quad w(0)=0=w(\pi).
    \end{aligned}
\end{align}

For Figure~\ref{gradtest}, we generated $1000\times 1000$ matrices randomly in~\subref{grad1000rand}  and used a discretization on a $1000\times 1000$ grid in~\subref{grad1000} as described above and used Algorithm~\ref{algo} to find an eigenvalue with minimal $\lambda$, i.e., with input index $(1,1)$. After solving 6 and 7 eigenvalue problems respectively as described in lines 4 and 6 of Algorithm~\ref{algo}, we found an eigenvalue with  an error of approximately $10^{-9}$ and $10^{-10}$ respectively, confirming the result of Theorem~\ref{convtheorem}.

For Figure~\ref{exp: random} we generated $100\times 100$ matrices satisfying Assumptions~\ref{Ass1} and~\ref{Ass2} as described above and use Algorithm~\ref{algo} to find every eigenvalue, i.e., for every input index $(i,j)\in\{1,\dots,100\}\times\{1,\dots,100\}$. The axis describe the input index of Algorithm~\ref{algo} and the greyscale describes the base ten logarithm of the error described above. An iteration is computing the solution of one eigenvalue problem in lines 4 and 6 of Algorithm~\ref{algo}. After 7 iterations the highest error was $4\cdot 10^{-8}$ and the smallest errors were in the range of machine precision.

In Figure~\ref{exp: ellipse} we used a discretization of the Helmholtz equation~\eqref{ellipseforexperiment} on a $100\times 100$ grid and repeated the experiment of Figure~\ref{exp: random} for the resulting matrices. Again, after 7 iterations the highest error was $3\cdot 10^{-8}$ and the smallest errors were in range of machine precision. These experiments suggest that Algorithm~\ref{algo} can indeed be used to find every eigenvalue of right definite two-parameter eigenvalue problems.

In Figure~\ref{timecomp}\subref{subfig: times} we compared the time of both Algorithm~\ref{algo} and the algorithm {\verb|twopareig|}\cite{twopareig} for computing every eigenvalue of a two-parameter eigenvalue problem of varying sizes $n=m$. We generated matrices satisfying the assumptions as above. Notice that Algorithm~\ref{algo} can be run in parallel for different input indices to further improve efficiency. The experiment was run on a Intel Xeon Gold 6144 at 3.5 GHz with 384 GB RAM. We used 8 cores with 16 threads. We chose to solve 10 small eigenvalue problems in Algorithm~\ref{algo}, corresponding to $k=1,\dots,5$ in line 2.
For larger $n$ our method was indeed faster, and the asymptotic slope  of the time -$n$ graph on a loglog scale is smaller, indicating that the asymptotic computational cost is lower. In Figure~\ref{timecomp}\subref{subfig: errors} we measure the sum of the $n^2$ errors as described above.  We observe that our method computed eigenvalues with  higher accuracy. 

Finally, we repeated the last experiment, but we chose $S_1=I_n=S_2$. This made the matrices $B_1,B_2,C_1$ and $C_2$ diagonal, which effectively transformed the generalized eigenvalue problems in Algorithm~\ref{algo} into ordinary eigenvalue problems to further improve efficiency. When solving for every eigenvalue of the two-parameter problem, we can perform the left and right actions of $(-C_1)^{-\frac{1}{2}}$ and  $C_2^{-\frac{1}{2}}$ respectively and afterwards diagonalize $B_1$ and $B_2$. This justifies this experiment. The results are depicted in Figure~\ref{timecompdiag}. In Figure~\ref{timecompdiag}\subref{subfig: timesdiag} we see that the alternating method is faster than the method  {\verb|twopareig|}\cite{twopareig} for even smaller $n$. Again our method is more accurate.

\section{Conclusion and outlook}

We presented a new method for computing eigenvalues of the two-parameter eigenvalue problem. Our approach only requires solving generalized eigenvalue problems of the size of the matrices of the two-parameter problem and can therefore reduce the complexity compared to conventional methods. Our method also uses a search for the eigenvalues by index, which makes it possible to find successive eigenvalues of the two-parameter eigenvalue problem without deflation.

So far the technique of our proof only established global convergence for extremal eigenvalues and under definiteness assumptions. The numerical experiments however indicated convergence for every eigenvalue. Proposition~\ref{prop: local convergence} gives insight into the local convergence, but a global convergence proof remains an open problem. Although there are many classes of two-parameter eigenvalue problems that satisfy the assumptions (eventually after performing an affine transformation), many other interesting applications do not. Therefore it would be important to investigate if a generalization to 
non-singular problems, i.e., the operator $M_0$ is invertible, as in~\cite{JacobiDavidson04} or even to the general case as in~\cite{Hochstenbach2019} is possible.

This paper relies heavily on the assumptions \ref{Ass1} and \ref{Ass2}. A natural question is to relax these conditions, in particular the definiteness condition \ref{Ass2}. Indeed under weaker assumptions (such as when the matrices are \emph{almost} definite, with a few eigenvalues of the opposite sign), using inertia laws~\cite{nakatsukasa2019inertia} one can show that the generalized eigenvalue problem~\eqref{matrixev}, and hence also~\eqref{2paraev}, has many real eigenvalues. It would be of interest to investigate the applicability of the results here in such situations.

Finally, another interesting generalization that could be considered is to multiparameter eigenvalue problems with more than 2 parameters. The eigenvectors then form rank-one tensors %of eigenvalues
 similar to \eqref{matrixev} \cite{Atkinson72}. Again an alternating approach as in \cite{Holtz2012} can be used, however our proof technique will not work and in practice the generalization will not easily assure convergence even for extremal indices. A similar approach using the Tensor-Train format is used in~\cite{TTmulti}.
%\section{Questions}
%
%\begin{enumerate}
%    \item How big/interesting is the class of PDE eigenvalue problems?
%    \item Are there other interesting applications?
%    \item What to compare to?
%    \item Can we prove convergence for non extremal eigenvalues?
%    \item Is there a way to extend the algorithm to the case, where definiteness conditions are less restrictive? Can we extend it to multi-parameter eigenvalue problems?
%\end{enumerate}

\bibliographystyle{plain}
\bibliography{main}

\end{document}